\documentclass[11pt,a4paper]{article}
\usepackage{amsmath,amssymb,amsthm,amscd,mathrsfs}
\usepackage[dvips]{graphicx}
\newcommand{\N}{\mathbb{N}}
\newcommand{\Z}{\mathbb{Z}}

\newcommand{\R}{\mathbb{R}}
\newcommand{\C}{\mathbb{C}} 

\newcommand{\g}{\mathfrak{g}}

\DeclareMathOperator{\supp}{supp}

\newtheorem{thm}{Theorem}[section]
\newtheorem{df}[thm]{Definition}
\newtheorem{prop}[thm]{Proposition}
\newtheorem{lem}[thm]{Lemma}
\newtheorem{cor}[thm]{Corollary}
\newtheorem{rem}[thm]{Remark}

\newtheorem{exa}[thm]{Example}
\numberwithin{equation}{section}

\title{Twisted Euler transform of differential equations with an irregular singular point}
\date{}
\author{Kazuki Hiroe\thanks{E-mail:\texttt{kazuki@ms.u-tokyo.ac.jp}}}

\begin{document}
\maketitle
\begin{center}
Graduate School of Mathematical Sciences, The University of Tokyo, 3-8-1 Komaba, Meguro-ku, Tokyo, 153-8914, Japan.
\end{center}
\begin{abstract}
In \cite{Katz}, N. Katz introduced the notion of the middle convolution on
 local systems. This can be seen as a generalization of the Euler
 transform of Fuchsian differential equations. In this paper, we
 consider  the generalization of the Euler transform, the twisted Euler
 transform, and apply this to differential equations with irregular singular
 points. In particular, for differential equations with an
 irregular singular point of irregular rank 2 at $x=\infty$, we
 describe explicitly changes of local datum caused by twisted
 Euler transforms.   Also we attach
 these differential equations to Kac-Moody Lie algebras and show that
 twisted Euler transforms correspond to the actions of Weyl groups of
 these Lie algebras. 
\end{abstract}
\section{Introduction}
For a function $f(x)$, the following integral
\[
I_{a}^{\lambda}f(x)=\frac{1}{\Gamma(\lambda)}\int_{a}^{x}(x-t)^{\lambda-1}f(t)\,dt
\]
is called the Riemann-Liouville integral for $a,\lambda\in\C$. If we take a function $f(x)=(x-a)^{\alpha}\phi(x)$ where $\alpha\in\C\backslash\Z_{<0}$ and $\phi(x)$ is a holomorphic function on a neighborhood of $x=a$ and $\phi(a)\neq 0$, then it is known that
\[
I^{-n}_{a}f(x)=\frac{d^{n}}{dx^{n}}f(x).
\]
Hence one can consider the Riemann-Liouville integral to be a fractional or complex powers of derivation $\partial=\frac{d}{dx}$. This may allow us to write $\partial^{\lambda}f(x)=I_{a}^{-\lambda}f(x)$ formally. 

Moreover one can show a generalization of the Leibniz rule,
\[
\partial^{\lambda}p(x)\psi(x)=\sum_{i=0}^{n}\begin{pmatrix}\lambda\\i\end{pmatrix}p^{(i)}(x)\partial^{\lambda-i}\phi(x),
\] 
if $p(x)$ is a polynomial of degree equal to or less than $n$. 

Now let us consider a differential operator with polynomial coefficients,
\[
P(x,\partial)=\sum_{i=0}^{n}a_{i}(x)\partial^{i}.
\]
The above Leibniz rule assures that $$\partial^{\lambda+m}P(x,\partial)\partial^{-\lambda}$$ gives the new differential operator with polynomial coefficients if we choose a suitable $m\in\Z$. 
Moreover if $f(x)$ satisfies $P(x,\partial)f(x)=0$ and $I^{-\lambda}_{a}f(x)$ is well-defined for some $a,\lambda\in \C$, then we can see that
\begin{align*}
\partial^{\lambda+m}P(x,\partial)\partial^{-\lambda}I_{a}^{-\lambda}f(x)&=\partial^{\lambda+m}P(x,\partial)\partial^{-\lambda+\lambda}f(x)\\
&=\partial^{-\lambda+m}P(x,\partial)f(x)\\
&=0.
\end{align*}
Hence $\partial^{\lambda}$ turns a differential equation with polynomial coefficients $P(x,\partial)u=0$ into a new differential equation with polynomial coefficients $Q(x,\partial)u=0$, and moreover a solution of $Q(x,\partial)u=0$ is given by a solution of $P(x,\partial)u=0$ if the Riemann-Liouville integral is well-defined. This correspondence of differential equations is called the Euler transform.

For example, let us take the differential equation of the Gauss hypergeometric function,
\begin{equation}\label{Gauss hyp}
x(1-x)\partial^{2}u+(\gamma-(\alpha+\beta+1)x)\partial u-\alpha\beta u=0.
\end{equation}
Then we can see that
\begin{align*}
&\partial^{-\beta}(x(1-x)\partial^{2}+(\gamma-(\alpha+\beta+1)x)\partial -\alpha\beta )\partial^{\beta-1}\\
&=x(1-x)\partial+((\gamma-\beta)-(\alpha-\beta+1)x).
\end{align*}
And it is not hard to see that the general solution of $x(1-x)\partial+((\gamma-\beta)-(\alpha-\beta+1)x)u=0$ is given by constant multiples of $x^{\beta-\gamma}(1-x)^{\alpha-\gamma+1}$. Hence solutions of $(\ref{Gauss hyp})$ are 
\[
I_{c}^{\beta-1}x^{\beta-\gamma}(1-x)^{\alpha-\gamma}=\frac{1}{\Gamma(-\beta)}\int_{c}^{x}t^{\beta-\gamma}(1-t)^{\alpha-\gamma}(x-t)^{-\beta}\,dt
\]
for $c=0,1,\infty$. 

This  argument tells us that by the Euler transform $\partial^{\beta-1}$, we can reduce $(\ref{Gauss hyp})$ to an ``easier'' one,
\[
x(1-x)\partial+((\gamma-\beta)-(\alpha-\beta+1)x).
\]
And solutions of $(\ref{Gauss hyp})$ can be obtained from this ``easy'' equation.

This can be applicable to differential equations with an irregular singular point. For example, let us consider the differential equations,
\begin{align}
x\partial^{2}u+(\gamma-x)\partial u-\alpha u=0,\label{Kummer}\\
\partial^{2}u-x\partial u+\alpha u=0.\label{Hermite}
\end{align}
The first one is  the differential equation of the Kummer confluence
hypergeometric function. The second one is the  one of the Hermite-Weber function.
Then we have
\begin{align*}
&\partial^{-\alpha}(x\partial^{2}+(\gamma-x)\partial-\alpha)\partial^{\alpha-1}\\
&=x\partial+((\gamma-\alpha)-x),
\end{align*}
and
\begin{align*}
&\partial^{\alpha}(\partial^{2}-x\partial+\alpha)\partial^{-\alpha-1}\\
&=\partial-x.
\end{align*}
Solutions of these differential equations are $x^{\alpha-\gamma}e^{x}$ and $e^{\frac{x^{2}}{2}}$ up to constant multiples. Hence solutions of $(\ref{Kummer})$ and $(\ref{Hermite})$ are given by
\begin{align*}
\frac{1}{\Gamma(-\alpha)}\int_{c}^{x}(x-t)^{-\alpha}t^{\alpha-\gamma}e^{t}\,dt
\end{align*}
for $c=0,\infty$, and
\begin{align*}
\frac{1}{\Gamma(\alpha)}\int_{-\infty}^{x}(x-t)^{\alpha}e^{\frac{t^{2}}{2}}\,dt&=\frac{1}{\Gamma(\alpha)}\int_{0}^{\infty}t^{\alpha}e^{\frac{(x-t)^{2}}{2}}\,dt\\
&=\frac{e^{\frac{x^{2}}{2}}}{\Gamma(\alpha)}\int_{0}^{\infty}e^{\frac{t^{2}}{2}-tx}t^{\alpha}\,dt.
\end{align*}

These facts tell us  that the Euler transform  is a good tool to study a solution of a differential equation from a solution of an easier differential equation. Then a question arises.
\vspace{1.5mm}

Can one always reduce differential equations to ``easy'' ones by Euler transforms? If not, 
find a good class of differential equations which can be reduced to easy one.
\vspace{1.5mm}

An answer of this question is known for Fuchsian differential equations. Let us consider a system of Fuchsian differential equations of the form,
\begin{equation}\label{schleginger normal form}
\frac{d}{dx}Y(x)=\sum_{i=1}^{r}\frac{A_{i}}{(x-c_{i})}Y(x),
\end{equation}
where $A_{i}$ are $n\times n$ complex matrices, $Y(x)$ is a $\C^{n}$-valued function and $c_{i}\in \C$. In this case, it is known that the Euler transform corresponds to the additive middle convolution (see \cite{D-W1} and \cite{D-W2}). The answer of our question is given by the following theorem.
\begin{thm}[Katz \cite{Katz}, Dettweiler-Reiter \cite{D-W1} \cite{D-W2}]\label{Theorem of Katz}
The tuple of $n\times n$ of matrices, $A_{1},\ldots,A_{r}$ is irreducible and linearly rigid, the differential equation $(\ref{schleginger normal form})$ can be reduced to
\[
\frac{d}{dx}y(x)=\sum_{i=1}^{r}\frac{a_{i}}{x-c_{i}}y(x)
\]
where $a_{i}\in\C$ by finite iterations of middle convolutions and additions.  
\end{thm}
Definitions of terminologies in this theorem can be found in the original papers \cite{D-W1} and \cite{D-W2}.

Our purpose of this paper is to extend this theorem to non-Fuchian differential equations. 

In this paper, we consider differential equations with polynomial
coefficients which have an irregular singular point of rank at most 2 at $x=\infty$ and some regular singular points in $\C$. And we give a generalization of Theorem \ref{Theorem of Katz}.

The organization of this paper is the following. In Section 2, we give a review of some operations on the Weyl algebra. These operations are provided by T. Oshima in \cite{O} to define the Euler transform in the strict way as an operation on the Weyl algebra. We use these operations to define a generalization of the Euler transform, the twisted Euler transform.

Because we treat differential equations with an irregular singular point, solutions of them have asymptotic expansions as formal power series. Hence in Section 3, we study the formal solutions of differential equations. We introduce the notion of semi-simple characteristic exponents here.

Section 4 is  one of the main parts of this paper. We focus on
differential equations which have an irregular singular point of rank 2 at
infinity and regular singular points in $\C$. We also assume that at the
irregular singular point, every formal solutions are ``normal type'',
i.e., formal solutions can be written by
$e^{p(x)}x^{-\mu}\sum_{s=0}^{\infty}c_{s}x^{-s}$ for polynomials $p(x)$. Equivalently to say, we assume that differential equations are unramified according to the terminology of the differential Galois theory. Then in Section 4, we investigate the way how the twisted Euler transform changes differential equations, in other words, how characteristic exponents at singular points are changed by the twisted Euler transform.

In Section 5, we see the relation between differential equations and Kac-Moody Lie algebras.
 In \cite{C-B1}, W. Crawley-Boevey find a correspondence of the systems
 of Fuchsian differential equations as $(\ref{schleginger normal form})$
 and representations of quivers and moreover Kac-Moody Lie algebras
 associated with these representations. And he solved the  existence problem of systems of differential equations, so-called Deligne-Simpson problem by using the theory of representations of quivers. An analogous work is done by P. Boalch in \cite{B}. Boalch generalizes the result of Crawley-Boevey for the cases which allow an irregular singular point of rank 2. And he also solved Deligne-Simpson problem with an irregular singular point of rank 2. 
 
 As an analogue of their works, we attach a Kac-Moody Lie algebra to the concerning differential equation. Moreover we show that twisted Euler transforms correspond to simple reflections on the Cartan subalgebra of this Lie algebra. By using this correspondence, we prove the following main theorem.
\begin{thm}
Let  us take $P(x,\partial)\in W[x,\xi]$ as in Definition \ref{local
 datum}. If $\mathrm{idx}\,P>0$, then $P(x,\partial)$ can be reduced to 
\[
(\partial-\alpha x-\beta)^{n}
\]
for some $\alpha,\beta\in \C$ and $n\in \Z_{>0}$ by finite iterations of
 twisted Euler transforms and additions at regular singular points. 
\end{thm} Notations used in this theorem are explained in the subsequent sections. 
 
In Section 6, we consider the confluence of Fuchsian differential equations. It is known that differential equations of Kummer confluent hypergeometric function and Hermite weber functions are obtained from the differential equation of the Gauss hypergeometric function by the limit transitions.
In Section 6,  as a generalization of this fact, we show the following.
\begin{thm}
Take  $P(x,\partial)\in W[x,\xi]$ as in Definition \ref{local datum}.
 If $\mathrm{idx}\,P>0$, then $P(x,\partial)$ can be obtained by the limit transition of a Fuchsian $Q(x,\partial)\in W[x,\xi]$ of $\mathrm{idx}\,Q=\mathrm{idx}\, P.$
\end{thm}

Meanwhile, in Appendix, we consider the differential equation which has regular singular point at $x=\infty$ and arbitrary singularities at any other points. And we give a necessary and sufficient condition to reduce the order of this differential equation by Euler transform $E(0,\mu)$. 
\begin{thm}
Let us take $P(x,\partial)\in W[x]$ which has regular singular point at $x=\infty$ and semi-simple exponents
\[
\{[\mu_{1}]_{n_{1}},\ldots,[\mu_{l}]_{n_{l}}\},
\] 
where $\sum_{i=1}^{l}n_{i}=n=\mathrm{ord}\,P$, $\mu_{i}\notin\Z$ and $\mu_{i}-\mu_{j}\notin\Z$ if $i\neq j$. Then we have
\[
\mathrm{ord}\,E(0,\mu_{i}-1)P(x,\partial)<\mathrm{ord}\,P
\]
if and only if 
\[
\deg P-\mathrm{ord}\, P<n_{i}.
\]

\end{thm}

 There are many other works about middle convolutions, equivalently to say, Euler transforms of differential equations with irregular singular points. T. Kawakami considers generalization of middle convolutions to systems of differential equations with irregular singular point which are called generalized Okubo systems in \cite{Kaw}. K. Takemura \cite{Take} and D. Yamakawa \cite{Yama} consider the middle convolutions for the system of the form
 \[
 \frac{d}{dx}Y(x)=\sum_{i=1}^{r}\sum_{j=1}^{k_{i}}\frac{A_{ij}}{(x-c_{i})^{j}}Y(x)
 \] 
where $A_{ij}$ are $n\times n$ complex matrices. In particular, Yamakawa discusses the reduction of the rank of a differential equation which is similar problem to ours.
\subsection*{Acknowledgement}
The author is very grateful to Professor Toshio Oshima who kindly taught the
author his theory of Euler transforms of Fuchsian differential
equations. The author also thanks Shinya Ishizaki. The author studied the theory
of middle convolutions and Euler transforms on his seminar under the
direction of Professor Oshima. Finally the author thanks Noriyuki Abe and
Ryosuke Kodera for the introduction of the Kac-Moody theory.
\section{Operations on localized Weyl algebra with parameters.}
In this section, some operations on localized Weyl algebra are
introduced. In \cite{O}, T. Oshima uses these operators to understand the Euler
transform and he constructs a theory of the Euler transform which
corresponds to the theory of additive middle convolutions for
Schleginger type systems of differential equations studied by Dettweiler
and Reiter in \cite{D-W1} and \cite{D-W2}.

The Weyl algebra $W[x]$ is the $\C$-algebra generated by $x$ and $\partial=\frac{d}{dx}$ with the relation,
\[
[\partial,x]=\partial x-x\partial=1.
\] 
The localization of the Weyl algebra is  $W(x)=\C(x)\otimes_{\C}W[x]$
where $\C(x)$ is the quotient field of $\C[x]$ the polynomial ring with
complex coefficients. For indeterminants
$\xi=(\xi_{1},\ldots,\xi_{n})$, we consider $\C(\xi)$ the field of
rational functions with complex coefficients and fix an algebraic closure $\overline{\C(\xi)}$ of $\C(\xi)$. Then we can define the Weyl algebra with parameters 
\[
W[x,\xi]=\overline{\C(\xi)}\otimes_{\C}W[x].
\]
We can also define the localized Weyl algebra with parameters
 $W(x,\xi)=\overline{C(\xi)}\otimes_{\C}W(x)$.

The element $P\in W(x,\xi)$ (resp. $P\in W[x,\xi]$) is uniquely written by
\[
P=\sum_{i=0}^{m}p_{i}(x;\xi)\partial^{i}
\]
with $p_{i}(x,\xi)\in\C(x,\xi)=\overline{\C(\xi)}\otimes_{\C}\C(x)$
(resp. $p_{i}(x,\xi)\in\C[x,\xi]=\overline{\C(\xi)}\otimes_{\C}\C[x]$)
for $i=0,\ldots,m$. According to this representation, the order of $P$ denoted by $\mathrm{ord}\,P$ is defined by the maximum integer $i$ such that $p_{i}(x,\xi)\neq 0$. We also define the degree of $P\in W[x;\xi]$ by maximum of  degrees of $p_{i}(x;\xi)$ as polynomials of $x$, and denoted by $\deg P$. 
\begin{df}
For $P(x,\partial)\in W(x,\xi)$ of $\mathrm{ord}\,P>0$, we say that
 $P(x,\partial)$ is irreducible when it is satisfied that if we can write $P=QR$ for some
 $Q,R\in W(x,\xi)$ then $\mathrm{ord}\, Q\cdot \mathrm{ord}\, R=0$.

\end{df}
We recall some operations on $W[x,\xi]$ and $W(x,\xi)$. Details of these operations can be found in \cite{O}.
\begin{df}[The Fourier-Laplace transform]
The Fourier-Laplace transform on $W[x,\xi]$ is the following algebra isomorphism
\[
\begin{array}{rcccc}
\mathcal{L}\colon&W[x,\xi]&\longrightarrow &W[x,\xi]&\\
&x&\longmapsto&-\partial&\\
&\partial&\longmapsto&x&\\
&\xi_{i}&\longmapsto &\xi_{i}&(i=1,\ldots,n).
\end{array}
\]
\end{df}
Sometimes we identify $P\in W(x,\xi)$ and $f(x)P$ for $f(x)\in
\C(x,\xi)$ because differential equations $P(x,\partial)u=0$ and $f(x)P(x,\partial)u=0$ can be identified.  If for $P,Q\in W(x,\xi)$, there exist $f(x)\in\C(x,\xi)$ such that $P=f(x)Q$, then we write
\[
P\sim Q.
\] 
\begin{df}[The reduced representative] The reduced representative of a nonzero element $P\in W[x,\xi]$ is an element of $\C(x,\xi)P\cap W[x,\xi]$ of the minimal degree and denoted by $\mathrm{R}P$.

\end{df}

\begin{df}
For $h\in \C(x,\xi)$, we define  the automorphism of $W(x,\xi)$ by
\[
\begin{array}{ccccc}
\mathrm{Adei}(h)\colon&W(x;\xi)&\longrightarrow &W(x;\xi)&\\
&x&\longmapsto&x&\\
&\partial&\longmapsto&\partial-h\\
&\xi_{i}&\longmapsto&\xi_{i}&(i=1,\ldots,n).
\end{array}
\]
\end{df}
\begin{df}[The addition at the singular point]
For $c\in\C$ and $f(\xi)\in\C(\xi)$ we call the operator
 $\mathrm{Adei}(\frac{f(\xi)}{x-c})$ the addition at the singular point
 $c$ and denote it by $\mathrm{Ad}((x-c)^{f(\xi)})$. We also define $\mathrm{RAd}((x-c)^{f(\xi)})=\mathrm{R}\circ\mathrm{Ad}((x-c)^{f(\xi)})$.
\end{df}
\begin{rem}\label{addition}
Let $f(x)$ and $g(x)$ be sufficiently many differentiable functions. Then the Leibniz rule tells us that
 \[
 \frac{d}{dx}f(x)g(x)=(f(x)\frac{d}{dx}+\frac{d}{dx}f(x))g(x).
 \]
 Hence for some $c,\lambda\in\C$ we have
 \[
 (x-c)^{\lambda}\frac{d}{dx}(x-c)^{-\lambda}g(x)=(\frac{d}{dx}-\frac{\lambda}{x-c})g(x).
 \]
 Therefore $\mathrm{Adei}(\frac{\lambda}{x-c})\frac{d}{dx}$ corresponds to
 \[
  (x-c)^{\lambda}\frac{d}{dx}(x-c)^{-\lambda}.
 \]
\end{rem}
\begin{df}[The $e^{p(x)}$-twisting]
For $p(x)\in\C[x]$, we define
\[
 \mathrm{Ade}({p(x)})=\mathrm{Adei}(p'(x))
\]
and call this the $e^{p(x)}$-twisting. Here $p'(x)=\frac{d}{dx}p(x)$.
\end{df}
\begin{rem}
As well as additions at singular points, the $e^{p(x)}$-twisting corresponds
 to the operation
\[
 \frac{d}{dx}\longmapsto e^{p(x)}\frac{d}{dx}e^{-p(x)}=(\frac{d}{dx}-p'(x)).
\]
\end{rem}
\begin{df}[The twisted Euler transform]
For $\alpha\in \C$ and $f(\xi)\in\C(\xi)$ we define the operator on $W[x,\xi]$,
\[
E(\alpha,f(\xi))=\mathcal{L}\circ\mathrm{RAd}((x+\alpha)^{-f(\xi)})\circ\mathcal{L}^{-1}\circ
\mathrm{R}.
\]
 For
 $\overline{\alpha}=(\alpha_{1},\ldots,\alpha_{m})\in \C^{m}$ and
 $F(\xi)=(f_{1}(\xi),\ldots,f_{m}(\xi))\in (\C(\xi))^{m}$, we
 define the operator
\[
E(p(x);\overline{\alpha},F(\xi))=\mathrm{Ade}(p(x))\circ\prod_{i=1}^{m}E(\alpha_{i},f_{i}(\xi))\circ\mathrm{Ade}(-p(x)) .
\]
We call $E(\alpha,f(\xi))$ and $\mathrm{E}(p(x);\overline{\alpha},F(\xi))$ twisted Euler transforms. In particular we call $\mathrm{E}(0,f(\xi))$ an Euler transform.
\end{df}
\begin{rem}
A twisted Euler transform $E(\alpha,f(\xi))$ corresponds to the  following integral transform
\begin{multline*}
\int_{-i\infty}^{i\infty}(y+\alpha)^{-f(\xi)}\int_{c}^{\infty}g(x)e^{-xy}\,dx\, e^{zy}\,dy\\
=\frac{1}{\Gamma(f(\xi))}\int_{c}^{z}g(x)(z-x)^{f(\xi)-1}e^{-\alpha(z-x)}\,dx.
\end{multline*}
\end{rem}
If we put $\alpha=0$, then this is a Riemann-Liouville integral. Hence
$E(0,f(\xi))$ can be seen as  the classical Euler transform.
\begin{rem}\label{Euler and twisted Euler}
A twisted Euler transform can be written as the composition of
 $e^{p(x)}$-twistings and an  Euler transform. We define the parallel displacement for $\alpha\in\C$ by
\[
\begin{array}{ccccc}
\mathrm{P}(h)\colon&W(x;\xi)&\longrightarrow &W(x;\xi)&\\
&x&\longmapsto&x-\alpha&\\
&\partial&\longmapsto&\partial\\
&\xi_{i}&\longmapsto&\xi_{i}&(i=1,\ldots,n).
\end{array}
\]
Then we can see that $\mathcal{L}^{-1}\mathrm{Ade}(\alpha x)=\mathrm{P}(-\alpha)\mathcal{L}^{-1}$ and $\mathrm{Ade}(-\alpha x)\mathcal{L}=\mathcal{L}\mathrm{R}(-\alpha).$ Hence we have
\[
E(\alpha,f(\xi))=\mathrm{Ade}(\alpha x)E(0,f(\xi))\mathrm{Ade}(-\alpha x).
\]

\end{rem}
\begin{prop}\label{mcirred}
We take an element $P(x,\partial)\in W(x,\xi)$ and assume
 that $P(x,\partial)$ is irreducible and  $\mathrm{ord}\,P>1 .$ Then
 we have followings.
\begin{enumerate}
\item We have 
\[
E(*,0)P\sim P
\]
where we can put any $\alpha\in\C$ into $*$.
\item If $E(\alpha,f(\xi))P$ is irreducible and $\mathrm{ord}\,E(\alpha,f(\xi))P>1$ for an $\alpha\in\C$ and an
      $f(\xi)\in\C(\xi)$, we have
\[
 E(\alpha,-f(\xi))E(\alpha,f(\xi))P=E(\alpha,0)P\sim P
\]
.
\end{enumerate}
\end{prop}
Before proving this proposition, we see the following lemma.
\begin{lem}\label{ueue}
Suppose that $P\in W[x,\xi]$ is irreducible. If $E(*,0)P=g(x)\in
 \C[x,\xi]$, then there exists $f(x)\in \C[x,\xi]$
 such that 
\[
 P\sim f(\partial)g(x)
\]
and the degree of $f(x)$ as the polynomial of $x$ is at most one.
\end{lem}
\begin{proof}
By the assumption, we have
\[
E(*,0)P=\mathcal{L}\mathrm{R}\mathcal{L}^{-1} \mathrm{R}P=g(x). 
\]
This implies that
\[
 g(\partial)=\mathcal{L}^{-1}g(x)=\mathrm{R}\mathcal{L}^{-1} \mathrm{R}P,
\]
that is, there exists an $f(x)\in\C[x,\xi]$ such that
\[
 \mathcal{L}^{-1} \mathrm{R}P=f(x)g(\partial).
\]
Equivalently we have
\[
 P\sim f(-\partial)g(x).
\]
By the irrducibility of $P$, $f(x)$ must be an irreducible polynomial, i.e.,
 its degree is at most one.   
\end{proof}
\begin{proof}[proof of Proposition \ref{mcirred}]
\end{proof}
Let us put $Q=E(*,0)P=\mathcal{L}
\mathrm{R}\mathcal{L}^{-1} \mathrm{R}P$. Then we have
\[
 \mathrm{R}\mathcal{L}^{-1}\mathrm{R}P=\mathcal{L}^{-1}Q. 
\]
Hence there exists an $f(x)\in \C[x,\xi]$ such that
\[
 \mathcal{L}^{-1}\mathrm{R}P=f(x)\mathcal{L}^{-1}Q,
\]
that is, we have
\[
 P\sim f(-\partial)Q.
\]
By Lemma \ref{ueue} and the assumption $\mathrm{ord}\,P>1$, we can see that  $\mathrm{ord}\,Q\neq 0$. Since $P$ is
irreducible, we have $f(\partial)\in \C$.  Hence we obtain 
\[
E(*;0)P\sim P.
\]

Let us recall that
\[
\mathrm{RAd}((x+\alpha)^{f(\xi)})\circ\mathrm{RAd}((x+\alpha)^{-f(\xi)})Q\sim Q
\]
for $Q\in W[x,\xi]$. 

By the assumption,  $E(\alpha,f(\xi))P$ is
irreducible and $\mathrm{ord}\,
E(\alpha,f(\xi))P>1$.  Hence we obtain
\begin{align*}
&E(\alpha,-f(\xi))E(\alpha,f(\xi))= E(\alpha,-f(\xi))E(*,0)E(\alpha,f(\xi))P\\
&=\mathcal{L}^{-1}\mathrm{RAd}((x+\alpha)^{f(\xi)})\mathrm{RAd}((x+\alpha)^{-f(\xi)})\mathcal{L}\circ \mathrm{R}P\\
&\sim P.
\end{align*}
\section{Formal solutions of differential equations}
We study formal solutions of differential equations with Laurent series
coefficients in this section. Although this section will contain many well-known
facts, we will give proofs of these for the completeness of the
paper. One of the purpose of this section is to understand when a
system of fundamental solutions has no logarithmic singularities even
though differences of characteristic exponents are integers. One of the
answer of this question is given by Oshima for Fuchsian differential
equations in \cite{O}.  We apply his result to formal
solutions of differential equations with Laurent series coefficients.
\subsection{Formal solutions at infinity.}
Let $\mathcal{K}$ be the field of Laurent series of $x^{-1}$, also write $\C[[x^{-1}]][x]$, and $\mathcal{K}\langle \partial\rangle$ the ring of differential operators with coefficients in $\mathcal{K}$.

Let us consider an element in $\mathcal{K}\langle\partial\rangle$,
\[
P(x,\partial)=a_{n}(x)\partial^{n}+a_{n-1}(x)\partial^{n-1}+\ldots +a_{1}(x)\partial+a_{0}(x)  
\] 
for $a_{i}(x)\in \C[[x^{-1}]][x]$. For
$a(x)=\sum_{n=-\infty}^{\infty}a_{n}x^{-n}\in \C[[x^{-1}]][x]$,
$v(a(x))$ denote the valuation of $a(x)$, i.e.,
$v(a(x))=\min\{s\in\Z\mid
a_{s}\neq 0\}$.
We consider ``normal'' formal solutions near $x=\infty$, i.e., solutions of the form $x^{\mu}f(x)$ for some $\mu\in \C$ and $f(x)\in \C[[x^{-1}]]$  $(f(\infty)\neq 0)$.

 The group ring generated by $\{x^{\mu}\mid \mu\in\C\}$ is denoted by $\C[x^{\mu}]$. Also the polynomial
ring of $\log x^{-1}$ with coefficients in $\C[[x^{-1}]]$ is denoted by $\C[[x^{-1}]][\log x^{-1}]$. We consider the tensor product $\C[x^{\mu}]\otimes_{\C}\C[[x^{-1}]][\log x^{-1}]$ as $\C$-algebras and we simply write these elements $x^{\mu}\otimes h(x)$ by $x^{\mu}h(x)$ for $\mu\in\C$ and $h(x)\in \C[[x^{-1}]][\log x^{-1}]$.

As an analogue of systems of fundamental solutions around regular
singular points, we consider subspaces of $\C[x^{\mu}]\otimes_{\C}\C[[x^{-1}]][\log x^{-1}]$ which are
spanned by formal power series with logarithmic terms. 
 
\begin{df}\label{normal}
Let us take $\mu_{1},\ldots \mu_{r}\in \C,\ \text{such that}\ \mu_{i}-\mu_{j}\notin\Z \ \text{for}\ i\neq j,$ increasing sequences of positive integers $0=n_{0}^{i}<\cdots<n_{l_{i}}^{i}$ for $\ i=1,\ldots,r$ and $m^{i}_{j}\in\Z_{\ge0}\ \text{for}\ i=1,\ldots,r,\,j=0,\ldots,l_{i}$. We put $h^{i}_{j}=\sum_{k=1}^{j}m^{i}_{k}$ and $n=\sum_{i=1}^{r}h^{i}_{l_{i}}$.

Then we consider the following $n$ functions, 
\[
f(\mu_{i}+n^{i}_{j};k)(x)=x^{\mu_{i}+n^{i}_{j}}\sum_{l=0}^{h^{i}_{j}+k}\begin{pmatrix}h^{i}_{j}+k\\ l\end{pmatrix}c^{ij}_{h^{i}_{j}+k-l}(x)(\log x^{-1})^{l},
\]
where $\begin{pmatrix}k\\l\end{pmatrix}$ are binomial coefficients,  $c^{ij}_{k}(x)\in \C[[x^{-1}]]$ and $c_{0}^{i0}(\infty)\neq 0$. We
 call the $\C$-vector subspace spanned by these functions the space of formal regular series. 
\end{df}

\begin{lem}\label{regular series}
We use the same notations as in Definition \ref{normal}.  Let $V$ be a $n$-dimensional space of formal regular series. For $i=1,\ldots,r$, we define $\C$-linear maps
\[
\begin{array}{cccc}
\Phi_{i}\colon &V&\longrightarrow &\C[x^{\mu}]\otimes_{\C}\C[[x^{-1}]][\log x^{-1}]\\
&f(x)&\mapsto &\partial\frac{f(x)}{f(\mu_{i};0)(x)}.\\
\end{array}
\]
Then each image $\Phi_{i}(V)$ is  $n-1$-dimensional space of formal regular series.
\end{lem}
\begin{proof}
We can see that 
\[
\partial f(\mu_{i}+n^{i}_{j};k)(x)=x^{\mu_{i}+n^{i}_{j}-1}\sum_{l=0}^{h^{i}_{j}+k}\begin{pmatrix}h^{i}_{j}+k\\ l\end{pmatrix}H^{ij}_{h^{i}_{j}+k-l}(x)(\log x^{-1})^{l},
\]
where 
\[
H^{ij}_{h^{i}_{j}+k-l}=\{((\mu_{i}+n_{k+1})+x\partial)c_{h^{i}_{j}+k-l}^{ij}(x)-(h^{i}_{j}+k-l)c_{h^{i}_{j}+k-l-1}^{ij}(x)\}.
\]
Also we can see that
\[
\partial \frac{1}{f(\mu_{i'};0)(x)}=x^{-\mu_{i'}-1}F^{i'}(x),
\]
where $F^{i'}(x)=((-\mu_{i'}+x\partial)c_{0}^{i'0}(x)^{-1}).$ Clearly $F^{i'}(\infty)\neq 0$. 

Hence we have
\begin{equation*}
\partial (\frac{f(\mu_{i}+n^{i}_{j};k)(x)}{f(\mu_{i'};k)(x)})=x^{\mu_{i}-\mu_{i'}+n^{i}_{j}-1}\sum_{l=0}^{h^{i}_{j}+k}\begin{pmatrix}h^{i}_{j}+k\\l\end{pmatrix}G_{h^{i}_{j}+k-l}^{ij}(x)(\log x^{-1})^{l}.
\end{equation*}
where 
\[
G_{k}^{ij}(x)=H_{k}^{ij}(x)f(\mu_{i'};0)(x)^{-1}+c_{k}^{ij}(x)F^{i'}(x).
\]
The image $\Phi_{i'}(V)$ is spanned by $\left\{\partial (\frac{f(\mu_{i}+n^{i}_{j};k)(x)}{f(\mu_{i'};k)(x)})\right\}_{i,j,k}$ .
Since $G_{0}^{i0}(x)=((\mu_{i}-\mu_{i'})+x\partial)c^{i0}_{0}(x)c^{i'0}_{0}(x)^{-1}$, we have $G_{0}^{i0}(\infty)\neq 0$ if $i\neq i'$. If $i=i'$, $G_{0}^{i0}(x)=0$. However, we have 
\[
\partial(\frac{f(\mu_{i'};1)(x)}{f(\mu_{i'};0)(x)})=\partial(c^{i'0}_{1}(x)c^{i'0}_{0}(x)^{-1}+\log x^{-1}).
\]
Hence we obtain that $G_{1}^{i0}(x)=(1+x\partial c^{i'0}_{1}(x)c^{i'0}_{0}(x)^{-1})$ and $G_{1}^{i0}(\infty)\neq 0$.  Hence $\Phi_{i'}(V)$ is the space of formal regular series.

\end{proof}
\begin{lem}\label{equiv}
Let $P(x,\partial)\in\mathcal{K}\langle\partial\rangle$ be 
\[
P(x,\partial)=a_{n}(x)\partial^{n}+a_{n-1}(x)\partial^{n-1}+\ldots +a_{1}(x)\partial+a_{0}(x), 
\] 
for $a_{i}(x)\in \C[[x^{-1}]][x]$. It can be written by
\[
P(x,\partial)=\sum_{s}^{\infty}x^{\rho-s}P_{s}(\partial),
\]
where $\rho=\max\{v(a_{i}(x))\mid i=0,\ldots,n\}$ and $P_{s}(x)\in \C[x]$.
Then the followings are equivalent
\begin{enumerate}
\item $P_{i}^{(j)}(0)=\partial^{j}P_{i}(x)|_{x=0}=0\ \text{for}\ i+j< l.$
\item $v(a_{l-i-1}(x))<\rho-i$\ $(i=0,\ldots,l-1)$.
\end{enumerate}
\end{lem}
\begin{proof}
The condition 1 is equivalent to that $P(x,\partial)$ is written by
\[
P(x,\partial)=x^{\rho}\partial^{l}P'_{0}(\partial)+x^{\rho-1}\partial^{l-1}P'_{1}(\partial)+\cdots +x^{\rho-l+1}\partial P'_{l-1}(\partial)+x^{\rho-l}P_{l}(\partial)+\cdots.
\]
Then the assertion is obvious.
\end{proof}
\begin{prop}\label{formal regular solutions}
Let $P(x,\partial)\in\mathcal{K}\langle\partial\rangle$ be 
\[
P(x,\partial)=a_{n}(x)\partial^{n}+a_{n-1}(x)\partial^{n-1}+\ldots +a_{1}(x)\partial+a_{0}(x), 
\] 
for $a_{i}(x)\in \C[[x^{-1}]][x]$. Also we write 
\[
P(x,\partial)=\sum_{s=0}^{\infty}x^{\rho-s}P_{s}(\partial)
\]
where $P_{s}(x)\in \C[x]$ and $\deg P_{s}\le n$ for $s=0,1,2,\ldots.$

For $l\le n$, we assume that there exist an $l$-dimensional space of formal regular series V whose elements are solutions of the differential equation $P(x,\partial)u=0$.
 
Then we have
\[
P_{i}^{(j)}(0)=0\ \text{for}\ i+j<l,
\]
equivalently
\[
v(a_{l-i-1}(x))<\rho-i\ (i=0,\ldots,l-1).
\]
\end{prop}
\begin{proof}
We prove this by induction on $l$. We assume that $P(x,\partial)$ has a formal solution $x^{\mu}h(x)$ for $\nu\in\C$ and $h(x)\in \C[[x^{-1}]]\ (h(\infty)\neq 0)$. Then the coefficient of $x^{\mu}$ of $P(x,\partial)x^{\mu}h(x)$, is $P_{0}(0)h(\infty).$ Hence  $P_{0}(0)=0.$

We assume that there exists a $k$-dimensional space $V$ of  formal regular series whose elements are solutions of  $P(x,\partial)u=0$. We take an element of $V$ without $\log$ term, i.e., an element $\phi(x)=x^{\mu}h(x)$ for $\mu\in\C$ and $h(x)\in \C[[x^{-1}]]\ (h(\infty)\neq 0)$. We consider a differential operetor $\tilde{P}(x,\partial)=\phi(x)^{-1}P(x,\partial)\phi(x)$. Then there exits $Q(x,\partial)\in\mathcal{K}\langle\partial\rangle$ such that $\tilde{P}(x,\partial)=Q(x,\partial)\partial$ because $1=\phi(x)^{-1}\phi(x)$ is a solution of $\tilde{P}(x,\partial)u=0$. Hence if we consider a $\C$-linear map
\[
\begin{array}{cccc}
\Phi\colon &V&\longrightarrow &\C[x^{\mu}]\otimes_{\C}\C[[x^{-1}]][\log x^{-1}]\\
&f(x)&\mapsto &\partial\frac{f(x)}{\phi(x)},\\
\end{array}
\]
elements in the image $\Phi(V)$ are formal solutions of $Q(x,\partial)$. By Lemma \ref{regular series}, $\Phi(V)$ is a $k-1$-dimensional space of formal regular series and elements of $\Phi(V)$ are formal solutions of $Q(x,\partial)u=0$ by the construction. Hence by the hypothesis of the induction we have
\[
Q_{i}^{j}(0)=0\ \text{for}\ i+j<k-1.
\]
If we write 
\[
Q(x,\partial)=\sum_{i=0}^{n-1}q_{i}(x)\partial^{i}
\]
for $q_{k}(x)\in\C[[x^{-1}]][x]\ (k=0,\ldots,n-1)$, the above condition is equivalent to
\begin{equation}\label{q-eq}
v(q_{k-i-2}(x))<\rho-i\ (i=0,\ldots,k-2)
\end{equation}
by Lemma \ref{equiv}. Also we can write
\[
P(x,\partial)=\sum_{i=0}^{n}p_{i}(x)\partial^{i}
\]
for $p_{i}(x)\in\C[[x^{-1}]][x]\ (i=0,\ldots,n)$. If we recall the Leibniz rule,\[
\partial^{i} f(x)g(x)=\sum_{j=0}^{i} \begin{pmatrix}i\\j\end{pmatrix}f^{(j)}(x)\partial^{i-j}g(x),
\]
the equation $\phi(x)P(x,\partial)\phi(x)^{-1}=Q(x,\partial)\partial$ says that\[
p_{n-i}(x)=\sum_{j=0}^{i}\begin{pmatrix}n-j\\i-j\end{pmatrix}\frac{\phi^{(i-j)}(x)}{\phi(x)}q_{n-j-1}(x)\ (i=0,\ldots,n).
\]
Here we put $q_{-1}(x)=0$. Since $v(\frac{\phi^{(i)}}{\phi(x)})\le-i$, valuations of $p_{i}(x)$ are bounded as follows,
\[
v(p_{i}(x))\le \max\{v(q_{j-1})-(j-i)\mid j=i,\ldots,n\}.
\] 
If we notice that 
\[
v(p_{i+1}(x))\le \max\{v(q_{j-1})-(j-i)+1\mid j=i+1,\ldots,n\},
\] 
then we have
\[
v(p_{i}(x))\le\max\{v(p_{i+1}(x))-1,v(q_{i-1}(x))\} \ (i=1,\ldots,n).
\]
Then above inequalities and $(\ref{q-eq})$ implies that
\begin{equation}
v(p_{k-i-1}(x))<\rho-i\ (i=0,\ldots,k-1).
\end{equation}

\end{proof}

\begin{df}[The characteristic equation]\label{characteristic equation at infinity}
Let us take an element of $\mathcal{K}\langle\partial\rangle$, 
\[
P(x,\partial)=\sum_{s=0}^{\infty}x^{\rho-s}P_{s}(\partial)
\]
where $P_{s}(x)\in \C[x]$ and $\deg P_{s}\le n$ for $s=0,1,2,\ldots.$
We consider following polynomials of  $\lambda$,
\[
p_{k}(\lambda,x)=\sum_{i=0}^{k}(\lambda+i+1)_{k-i}\frac{1}{(k-i)!}P_{i}^{(k-i)}(x)
\]
for $k=1,\ldots,n$. Here we put $(\lambda)_{l}=\lambda(\lambda+1)\cdots(\lambda+(l-1))$ for $l=1,2,\ldots$ and $(\lambda)_{0}=1$. If $P_{0}^{(k)}(0)\neq 0$, then $p_{k}(\lambda,0)$ is a polynomial of $\lambda$ of degree $k$. 
Hence if $P_{0}^{(m)}(0)\neq 0$ and $P^{(i)}_{j}(0)=0$ for $i+j<m$, we call 
\[
p_{m}(\lambda-m,0)=0
\]
 the characteristic equation of $P(x,\partial)$.
\end{df}

\begin{lem}\label{frobenious method}
Let us take $P(x,\partial)$ as in Definition \ref{characteristic equation at infinity}.
Suppose that $P_{i}^{(j)}(0)=0$ for $i+j<m$ and $P_{0}^{(m)}(0)\neq 0$ for an $m\le n$. 

Also we assume that there exist $\mu_{1},\ldots \mu_{r}\in \C,\ \text{such that}\ \mu_{i}-\mu_{j}\notin\Z \ \text{for}\ i\neq j,$ increasing sequences of positive integers $0=n_{0}^{i}<\cdots<n_{l_{i}}^{i}$ for $\ i=1,\ldots,r$ and $m^{i}_{j}\in\Z_{>0}\ \text{for}\ i=1,\ldots,r,\,j=0,\ldots,l_{i}$ such that $m=\sum_{i=1}^{r}\sum_{j=0}^{l_{i}}m^{i}_{j}$. Also assume that  solutions of the characteristic equation 
\[
p_{m}(\lambda-m,0)=0
\]
are $\mu_{i}+n^{i}_{j}$ with multiplicities $m^{i}_{j}$ for $i=1,\ldots,r$, $j=0,\ldots,l_{i}$.

Then there exist following $m$ formal solutions of $P(x,\partial)u=0$,
\[
f(\mu_{i}+n^{i}_{j};k)(x)=x^{\mu_{i}+n^{i}_{j}}\sum_{l=0}^{h^{i}_{j}+k}\begin{pmatrix}h^{i}_{j}+k\\ l\end{pmatrix}c^{ij}_{h^{i}_{j}+k-l}(x)(\log x)^{l},
\]
where $c^{ij}_{k}(x)\in \C[[x^{-1}]]$, $c_{0}^{i0}(\infty)\neq 0$ and $h^{i}_{j}=\sum_{k=1}^{j}m^{i}_{k}$.
\end{lem}
\begin{proof}
Suppose that $x^{\mu}\sum_{s=0}^{\infty}c_{x}x^{-s}$ $(c_{0}\neq 0)$ is
 a formal solution of $P(x,\partial) u=\sum_{s=0}^{\infty}x^{\rho-s}P_{s}(\partial)u=0$. Then we have the following equations,
\begin{equation}\label{char}
\sum_{k=0}^{i}c_{i-k}p_{k}(\mu-i,0)=0,
\end{equation}
for $i=0,1,\ldots$. The assumption implies $p_{k}(\lambda,0)$ for $k<m$ are identically zero and  $p_{m}(\lambda,0)$ is nonzero polynomial of $\lambda$ of degree $m$. Hence this is a direct consequence of the Frobenius method.
\end{proof}

\begin{prop}\label{strictly normal}
Let us take $P(x,\partial)\in \mathcal{K}\langle\partial\rangle$ as in Definition \ref{characteristic equation at infinity}. 
Suppose that $P_{i}^{(j)}(0)=0$ for $i+j<m$ and $P_{0}^{(m)}(0)\neq0$. We take $\mu_{1},\ldots,\mu_{r}\in \C$ such that $\mu_{i}-\mu_{j}\notin\Z$ $(i\neq j)$. Then the followings are equivalent,
\begin{enumerate}
\item There exist $m$ formal series,
\[
f_{ij}(x)=x^{-\mu_{i}-j}+x^{-\mu_{i}-m_{i}}h_{ij}(x)
\]
where $h_{ij}(x)\in
 \C[[x^{-1}]]$ for $i=1,\ldots,l$ and $j=0,\ldots,m_{i}-1$ and these are
      solutions of $P(x,\partial)u=0$.
\item
\[p_{k}(-\mu_{i}-j-k;0)=0
\]
 for 
\begin{gather*}i=1,\ldots,r,\\
j=0,\ldots,m_{i}-1,\\
k=m,m+1,\ldots,m+m_{i}-j-1.
\end{gather*}
\end{enumerate}
\end{prop}
\begin{proof}
First we assume that 1 is true. We consider the equation 
\[
P(x,\partial)x^{-\mu_{i}-j}(c_{0}+\sum_{s=m_{i}-j}^{\infty}c_{s}x^{-s})=0
\]
where $i\in\{1,\ldots,r\}$, $j\in\{0,\ldots,m_{i}-1\}$ and $c_{0}\neq 0$. Since $c_{s}=0$ for $1<s<m_{i}-j$, the equation $(\ref{char})$ tells us that
\[
c_{0}p_{k}(-\mu_{i}-j-k;0)=0,
\]
 for  $m\le k<m_{i}-j+m$. If the condition 2 is false, $p_{k}(-\mu_{i}-j-k;0)\neq 0$ for some $k$. This contradicts $c_{0}\neq 0$. 
 
Conversely we assume that 2 is true. Let us consider the equation $P(x,\partial)x^{-\mu_{i}-j}(c_{0}+\sum_{s=m_{i}-j}^{\infty}c_{s}x^{-s})=0$. Then the equation $(\ref{char})$ implies
\[
c_{0}p_{k}(-\mu_{i}-j-k;0)=0,
\]
for  $n\le k<m_{i}-j+n$ and 
\begin{multline*}
\sum_{l=0}^{k} c_{m_{i}-j+l}p_{m+(k-l)}(-\mu_{i}-m_{i}-k-m,0)\\
+c_{0}p_{m_{i}-j+m+k}(-\mu_{i}-m_{i}-k-m,0)=0
\end{multline*}
for $k=0,1,\ldots$. By the assumption 2, we can choose $c_{0}\neq 0$. Then if $c_{0}$ is determined, other coefficients $c_{s}$ for $s=m_{i}-j,m_{i}-j+1,\ldots$ are determined inductively. We can put $c_{0}=1$. Then we can choose $x^{-\mu_{i}-j}(1+\sum_{s=m_{i}-j}^{\infty}c_{s}x^{-s})$ as a formal solution of $P(x,\partial)f(x)=0$ for all $i\in\{1,\ldots,r\}$, $j\in\{0,\ldots,m_{i}-1\}$.
\end{proof}
\begin{df}[Semi-simple characteristic exponents]
Let us take a differential operator
\[
P(x,\partial)=\sum_{s=0}^{\infty}x^{\rho-s}P_{s}(\partial)
\]
where $P_{s}(x)\in \C[x]$ and $\deg P_{s}\le n$ for $s=0,1,2,\ldots.$
 Suppose that $P_{i}^{(j)}(0)=0$ for $i+j<m$ and
 $P_{0}^{(m)}(0)\neq0$. Also suppose that  there exit $\mu_{1},\ldots,\mu_{r}\in \C$
 such that $\mu_{i}-\mu_{j}\notin\Z$ $(i\neq j)$ and these satisfy 
\[p_{k}(-\mu_{i}-j-k,0)=0
\]
 for 
\begin{gather*}i=1,\ldots,r,\\
j=0,\ldots,m_{i}-1,\\
k=m,m+1,\ldots,m+m_{i}-j-1.
\end{gather*} Here $m=\sum_{i=1}^{r}m_{i}$. Then we say that $P(x,\partial)$ has semi-simple characteristic exponents
\[
\{\mu_{1},\mu_{1}+1,\ldots,\mu_{1}+m_{1}-1,\ldots,\mu_{r},\mu_{r}+1,\ldots,\mu_{r}+m_{r}-1\},
\] 
at $x=\infty$.
By using the notation $[\mu]_{m}=\{\mu,\mu+1,\ldots\mu_{m-1}\}$, we write
\begin{multline*}
\{[\mu_{1}]_{m_{1}},\ldots,[\mu_{r}]_{m_{r}}\}\\
=\{\mu_{1},\mu_{1}+1,\ldots,\mu_{1}+m_{1}-1,\ldots,\mu_{r},\mu_{r}+1,\ldots,\mu_{r}+m_{r}-1\}
\end{multline*}
shortly.
\end{df}
Let us recall
that 
\[
 e^{-p(x)}\partial e^{p(x)}=\partial+p'(x).
\]
For a $P(x,\partial)\in \mathcal{K}\langle\partial\rangle$, the differential equation
$P(x,\partial)u=0$ has a formal solution 
\[
 e^{p(x)}x^{-\nu}\sum_{i=0}c_{s}x^{-s}
\]
if and only if 
the differential equation $P(x,\partial+p'(x))$ has a formal solution
\[
 x^{-\nu}\sum_{i=0}c_{s}x^{-s}.
\]
\begin{df}[$e^{p(x)}$-twisted semi-simple characteristic  exponents]
For $P(x,\partial)\in \mathcal{K}\langle\partial\rangle$ and $p(x)\in\C[x]$, we say that the differential equation
 $P(x,\partial)u=0$ has $e^{p(x)}$-twisted  semi-simple exponents
\[
 \{[\mu_{1}]_{m_{1}},\ldots,[\mu_{l}]_{m_{l}}\}
\]
at $x=\infty$ where $\mu_{i}\in \C$, $m_{i}\in \N$ and $m=\sum_{i=1}^{l}m_{i}$, if 
the differential equation $P(x,\partial+p'(x))$ has the same semi-simple
 exponents at $x=\infty$.
\end{df}
\begin{prop}\label{addition and exponents at infinity}
Suppose that  $P(x,\partial)\in W[x]$ has  $e^{p(x)}$-twisted semi-simple exponents
\[
 \{[\mu_{1}]_{m_{1}},\ldots,[\mu_{l}]_{m_{l}}\}
\]
at $x=\infty$ where $\mu_{i}\in \C$, $m_{i}\in \N$ and $m=\sum_{i=1}^{l}m_{i}$.
\begin{enumerate}
\item For $\alpha,\nu\in\C$, the differential equation
      $\mathrm{Ad}((x-\alpha)^{\nu})P(x,\partial)u(x)=0$ has $e^{p(x)}$-twisted semi-simple exponents
\[
 \{[\mu_{1}-\nu]_{m_{1}},\ldots,[\mu_{l}-\nu]_{m_{l}}\}
\]
at $x=\infty$.
\item For $q(x)\in\C$, the differential equation
      $\mathrm{Ade}(q(x))P(x,\partial)$ has 
      $e^{p(x)+q(x)}$-twisted semi-simple exponents
\[
 \{[\mu_{1}]_{m_{1}},\ldots,[\mu_{l}]_{m_{l}}\}
\]
at $x=\infty$.
\end{enumerate}
\end{prop} 
\begin{proof}
If we recall that $(x-\alpha)^{\nu}$ can be written as
 $(x-\alpha)^{\nu}=x^{-\nu}\sum_{i=0}^{\infty}c_{i}x^{-i}$, the first
 assertion follows from the same argument as Proposition \ref{addition
 and exponents}. The second assertion easily follows from
\[
 \mathrm{Ade}(q(x))P(x,\partial)=\exp(q(x))P(x,\partial)\exp(-q(x)).
\] 
\end{proof}
\subsection{A review of regular singularity}
Let us take a differential operator $P(x,\partial)\in W[x]$. If
$P(x,\partial)$ has a regular singular point  at $x=c\in \C$, then we can write
\[
P(x,\partial)=\sum_{i=0}^n(x-c)^{n-i}a_{i}(x)\partial^{n-i}
\]
where $a_{0}(c)\neq 0$.  We consider a polynomial of $\nu$
\[
f_{c}(x,\nu)=\sum_{j=0}^{n}\frac{a_{j}(x)}{a_{0}(x)}\nu(\nu-1)\cdots(\nu-(n-j)+1).
\] 
Since $x=c$ is a regular singular point of $P(x,\partial)$, $f_{c}(x,\nu)$ is holomorphic at $x=c$. Hence we have the Taylor expansion
\begin{equation}\label{characteristic polynomial regular}
f_{c}(x,\nu)=\sum_{k=0}^{\infty}f^{c}_{k}(\nu)(x-c)^{k},
\end{equation}
where $f^{c}_{k}(\nu)$ are polynomials of $\rho$. Then a power series $g(\nu,x)=(x-c)^{\nu}\sum_{k=0}^{\infty}d_{k}(x-c)^{k}$ satisfies $P(x,\partial)g(\nu,x)=0$ if and only if  equations
\begin{equation}
\sum_{k=0}^{l}c_{l-k}f^{c}_{k}(\nu+(l-k))=0
\end{equation}
 are satisfied for $l=0,1,\ldots.$ We call the equation $f^{c}_{0}(\rho)=0$ the characteristic equation at regular singular point $x=c$.

Then we have a similar result to Proposition \ref{strictly normal}.
\begin{prop}[Oshima \cite{O}]\label{strict regular}
Let us take $P(x,\partial)\in W[x]$ which has a regular singular point $x=c$. We write
\[
P(x,\partial)=\sum_{i=0}^{n}(x-c)^{n-i}a_{i}(x)\partial^{n-i}
\]
where $a_{i}(x)\in \C[x]$ and $a_{0}(c)\neq 0$. Then we can define polynomials $f^{c}_{k}(\nu)$ for $k=0,1,2,\ldots $ as $(\ref{characteristic polynomial regular})$. Then the followings are equivalent.
\begin{enumerate}
\item There exist $\mu_{1},\ldots,\mu_{l}\in \C$ such that $\mu_{i}-\mu_{j}\notin\Z$ if $i\neq j$. The following $n$ functions are solutions of $P(x,\partial)u=0$, 
\[
g_{ij}(x)=(x-c)^{\mu_{i}+j}+(x-c)^{\mu_{i}+m_{i}}h_{ij}(x-c)
\] 
for $i=1,\ldots,l$ and $j=0,\ldots,m_{i}-1$. Here $h_{ij}(x)\in \C[[x]]$ and $n=\sum_{i=1}^{l}m_{i}$.
\item There exist $\mu_{1},\ldots,\mu_{l}\in \C$ such that $\mu_{i}-\mu_{j}\notin\Z$ if $i\neq j$. For these $\mu_{i}$, we have
\[
f_{k}(\mu_{i}-j)=0
\]
for $i=1,\ldots,l$, $j=0,\ldots,m_{i}-1$ and $k=0,\ldots,m_{i}-j-1$.
\end{enumerate}
\end{prop}

\begin{prop}[Oshima \cite{O}]\label{reduce degree}
Let us take $P(x,\partial)\in W[x]$ which has a regular singular point $x=c$. We write
\[
P(x,\partial)=\sum_{i=0}^{n}(x-c)^{n-i}a_{i}(x)\partial^{n-i}
\]
where $a_{i}(x)\in \C[x]$ and $a_{0}(c)\neq 0$. 
Then the followings are equivalent.
\begin{enumerate}
\item There exist  $m$ functions,
\[
g_{i}(x)=(x-c)^{i}+(x-c)^{m}h_{i}(x-c)
\] 
for $i=0,\ldots,m-1$ and these are solutions of $P(x,\partial)u=0$. Here $h_{i}(x)\in \C[[x]]$.
\item There exist $Q(x,\partial)\in W[x]$ such that
\[
P(x,\partial)=(x-c)^{m}Q(x,\partial).
\]

\end{enumerate}
\end{prop}
\begin{proof}
Although the proof of this proposition can be found in \cite{O}, we prove
 this for the completeness. 
Suppose that 2 is true. We notice that $Q(x,\partial)(x-c)^{i}$ are holomorphic at $x=c$ for any $i\in \Z_{\ge0}$. Hence if we write $Q(x,\partial)(x-c)^{i}=h_{i}(x)\in\C[[x]]$, then
we have
\[
P(x,\partial)(x-c)^{i}=(x-c)^{m}Q(x,\partial)(x-c)^{i}=(x-c)^{m}h_{i}(x).
\]
Conversely, if 1 is true, there exist $h_{i}(x)\in \C[[x]]$ such that
\[
P(x,\partial)(x-c)^{i}=\sum_{j=0}^{n}(x-c)^{n-j}a_{j}(x)\partial^{n-j}(x-c)^{i}=(x-c)^{m}h_{i}(x),
\]
for $i=0,\ldots,m-1$. If $i=0$, we have
\[
a_{0}(x)=(x-c)((x-c)^{m-1}h_{0}(x)-\sum_{j=1}^{n}(x-c)^{n-i-1}a_{j}(x)).
\]
Hence $P(x,\partial)=(x-c)Q_{0}(x,\partial)$ for a $Q(x,\partial)\in W[x]$. If $i=1$, we have
\[
(x-c)(a_{1}(x)+a_{0}(x))=(x-c)^{2}((x-c)^{m-2}-\sum_{j=2}^{n}(x-c)^{n-j-2}a_{j}(x)).
\]
Hence $P(x,\partial)=(x-c)^{2}Q_{1}(x,\partial)$ for a $Q_{1}(x,\partial)\in W[x].$ We can iterate these for $i=0,1,\ldots, m-1$.
\end{proof}
We can define semi-simple characteristic exponents at regular singular
points as we define for formal solutions.
\begin{df}
Let us take  $P(x,\partial)\in W[x]$ which has a regular singular point $x=c$. We write
\[
P(x,\partial)=\sum_{i=0}^{n}(x-c)^{n-i}a_{i}(x)\partial^{n-i}
\]
where $a_{i}(x)\in \C[x]$ and $a_{0}(c)\neq 0$. Then we can define polynomials $f^{c}_{k}(\nu)$ for $k=0,1,2,\ldots $ as $(\ref{characteristic polynomial regular})$.   If there exist $\mu_{1},\ldots,\mu_{l}\in \C$ such that $\mu_{i}-\mu_{j}\notin\Z$ $(i\neq j)$ and we have 
\[
f_{k}(\mu_{i}-j)=0
\]
for $i=1,\ldots,l$, $j=0,\ldots,m_{i}-1$ and $k=0,\ldots,m_{i}-j-1$, then we say that $P(x,\partial)$ has  semi-simple characteristic exponents
\[
\{[\mu_{1}]_{m_{1}},\ldots,[\mu_{l}]_{m_{l}}\}
\]
at $x=c$.
\end{df}

\begin{prop}\label{addition and exponents}
Let $P(x,\partial)\in W[x]$ has a regular singular point $x=c$ and semi-simple exponents  
\[
 \{[\rho_{1}]_{m_{1}},\ldots,[\rho_{l}]_{m_{l}}\}.
\]
\begin{enumerate}
\item For $\nu\in \C$, the differential equation
      $\mathrm{Ad}((x-c)^{\nu})P(x,\partial)u=0$ has semi-simple
      exponents
\[
  \{[\rho_{1}+\nu]_{m_{1}},\ldots,[\rho_{l}+\nu]_{m_{l}}\}
\]
at $x=c$.
\item If $\alpha\neq c$, then the  addition at $x=\alpha$ does not change the
      set of exponents at $x=c$ of $P(x,\partial)$, that is, $\mathrm{Ad}((x-\alpha)^{\nu})P(x,\partial)$
      has semi-simple exponents
\[
  \{[\rho_{1}]_{m_{1}},\ldots,[\rho_{l}]_{m_{l}}\}
\]
at $x=c$ as well.
\item For $p(x)\in\C$, the set of exponents at $x=c$ of $P(x,\partial)$
      are not changed by $\mathrm{Ade}(p(x))$, that is, $\mathrm{Ade}(p(x))P(x,\partial)$ has semi-simple exponents
\[
  \{[\rho_{1}]_{m_{1}},\ldots,[\rho_{l}]_{m_{l}}\}
\]
at $x=c$ as well.
\end{enumerate}
\end{prop}
\begin{proof}
If a function $u(x)$ is a solution of $P(x,\partial)u(x)=0$, then for
 $\alpha,\nu\in\C$, the function $(x-\alpha)^{\nu}u(x)$ satisfies that
\begin{multline*}
\mathrm{Ad}((x-\alpha)^{\nu})P(x,\partial)(x-\alpha)^{\nu}u(x)\\
=(x-\alpha)^{\nu}P(x,\partial)(x-\alpha)^{-\nu}(x-\alpha)^{\nu}u(x)\\
=(x-\alpha)^{\nu}P(x,\partial)u(x)=0.
\end{multline*}
Hence if $\alpha=c$ and $u(x)=(x-c)^{\rho}\sum_{i=0}^{\infty}c_{i}x^{i}$
 is a solution of $P(x,\partial)u(x)=0$ around $x=c$, then $(x-c)^{\nu}u(x)=(x-c)^{\rho+\nu}\sum_{i=0}^{\infty}d_{i}x^{i}$
 is a solution of $\mathrm{Ad}((x-c)^{\nu})P(x,\partial)v(x)=0$.

 On the
 other hand, if $\alpha\neq c$, the function $(x-\alpha)^{\nu}$ is
 holomorphic at $x=c$. Hence we can write the Taylor expansion
 $(x-\alpha)^{\nu}=\sum_{i=0}^{\infty}e_{i}x^{i}$. This implies that
 $(x-\alpha)^{\nu}u(x)=\sum_{i=0}^{\infty}e_{i}x^{i}(x-c)^{\rho}\sum_{j=0}^{\infty}c_{j}x^{i}=(x-c)^{\rho}\sum_{i=0}^{\infty}f_{i}x^{i}$. Therefore
 exponents does not change.

For the final assertion, we recall that
\[
 \mathrm{Ade}(p(x))P(x,\partial)=\exp(p(x))P(x,\partial)\exp{(-p(x))},
\]
and $\exp((p(x))$ is holomorphic on $\C$. By the same argument as
 above, the final assertion follows.
\end{proof}

\section{Differential equations of irregular rank $2$ at infinity}
The twisted Euler transform turns  $P(x,\partial)\in W[x,\xi]$ into the
other $Q(x,\partial)\in W[x,\xi]$. The question is  how this
transformation changes local datum of these differential operators. We
focus on the case of  the rank of irregular singularity at most 2 and give
 explicit descriptions about the changes of characteristic exponents by
twisted Euler transforms.
\subsection{The rank of irregular singularity, the Newton polygon}
Let us recall the notions, the rank of irregular singularity and the
Newton polygon of a differential operator.
 
\begin{df}[The rank of irregular singularity at infinity]\label{irregular}
Let us consider a linear differential equation,
\begin{equation}\label{irreg}
[x^{n}\partial^{n}+a_{1}(x)x^{n-1}\partial^{n-1}+a_{2}(x)x^{n-2}\partial^{n-2}+\cdots +a_{n}(x)]f(x)=0.
\end{equation}
Here coefficients $a_{i}(x)$ are Laurent series,
\[
a_{i}(x)=x^{m_{i}}\sum_{k=0}^{\infty}a_{k}^{i}x^{-k}\ (a_{0}^{i}\neq 0),
\]
where $m_{i}\in\Z$ for $i=1,\ldots,n$. 

The rank of irregular singularity at infinity of $(\ref{irreg})$ is the number defined by
\[
q=\max\{\frac{m_{i}}{i}\mid i=1,\ldots,n\}.
\]
\end{df}
Let us take  $P(x,\partial)=\sum a_{i}(x)\partial^{i}\in\mathcal{K}\langle\partial\rangle$. Every $a_{i}(x)\partial^{i}$ associate the point
$(i,i-v(a_{i}))$ of $\N\times \Z$. Then we define the Newton polygon $N(P)$
of $P$ to be the convex hull of the set
\[
 \bigcup_{i}\{(x,y)\in \R^{2}\mid x\le i,\, y\ge i-v(a_{i})\}.
\]
Let  $\{s_{i}=(u_{i},v_{i})\}_{0\le i \le p}$  be the set of vertices of this polygon such that
$0=u_{0}<u_{1}<\cdots <u_{p}=n\ (n=\mathrm{ord}\, P)$. Slopes of the edge connecting $s_{i}$ and $s_{i-1}$ are
\[
 \lambda_{i}=\frac{v_{i}-v_{i-1}}{u_{i}-u_{i-1}}
\]
 for $i=0,\ldots,p-1$. Clearly we have $\lambda_{1}<\lambda_{2}<\cdots<\lambda_{p}$. We define lengths $L_{i}$ of segments $[s_{i-1},s_{i}]$ by $L_{i}=u_{i}-u_{i-1}.$
 We note that $\lambda_{p}$ corresponds to the irregular
 rank at infinity of $P$. We refer \cite{Mal},\cite{Ram} and
 $\cite{Tour}$ for further things about Newton polygons.
 
 \begin{rem}\label{Newton polygon and formal solutions}
 Let us consider the Newton polygon of $P(x,\partial)$ whose vertices
  $s_{0},s_{1},\ldots,s_{r}$, slopes $\lambda_{1}<\lambda_{2}<\cdots
  <\lambda_{r}$ and lengths of segments are $L_{1},\ldots,L_{r}$. Then
  for $1\le i\le r$, there exist $q\in\Z_{>0}$ and  $L_{i}$ linearly independent formal solutions of $P(x,\partial)u=0$,
 \[
 f^{ij}_{k}(x)=e^{p^{ij}(x)}x^{\mu^{ij}_{k}}h^{ij}_{k}(x^{\frac{1}{q}})
 \]
  for $q\in\Z_{>0}$ $h^{ij}_{k}(x)\in\C[[x^{-1}]][\log x^{-1}]$.
 Here 
 \[
 p^{ij}(x)=\sum_{l=0}^{r^{ij}}a^{ij}(l)x^{\frac{r^{i}-l}{q}}
 \]
and $\frac{r_{i}}{q}=\lambda_{i}$.
\end{rem}

\subsection{Fourier-Laplace transform}
As is well-known, the Fourier-Laplace transform exchanges regular singular points on $\C$ and irregular singular point of rank 1 at infinity. Let us see the way how the  Fourier-Laplace transform changes ranks of irregular
singularities and characteristic exponents of a differential operators of
irregular rank at most 2.
\begin{prop}\label{oouue}
Let us take  $P(x,\partial)\in W[x]$ of $\deg P=N$ and $\mu_{1},\ldots,\mu_{l}\in \C$ such that $\mu_{i}\notin \Z$ and $\mu_{i}-\mu_{j}\notin \Z$ $(i\neq j)$. Then the followings are
 equivalent.
\begin{enumerate}
\item The differential equation $P(x,\partial)u=0$ has semi-simple exponents
\[
 \{[\mu_{1}]_{m_{1}},\ldots,[\mu_{l}]_{m_{l}}\}
\]
at $x=\infty.$
Here $m=\sum_{i=1}^{l}m_{i}.$
\item For the Fourier-Laplace transform $P(-\partial,x)$ of $P(x,\partial)$ has
 regular singular point  at $x=0$ and $P(-\partial,x)u=0$ has the semi-simple characteristic exponents,
 \[
 \{[0]_{N-m}[\mu_{1}-1]_{m_{1}},\ldots,[\mu_{l}-1]_{m_{l}}\}
\]
at $x=0$.
\end{enumerate}
\end{prop}
\begin{proof}
Suppose that 1 is true. From the assumption $\mu_{i}\notin \Z$, we can see $N\ge m$. If $N<m$,  we can write 
\[
P(x,\partial)=\sum_{s=0}^{N}x^{N-s}P_{s}(\partial)\partial^{m-s}=Q(x,\partial)\partial^{N-m}
\]
for $Q(x,\partial)\in W[x]$. Then polynomials $\sum_{i=0}^{N-m-1}a_{i}x$ for $a_{i}\in\C$ satisfy $P(x,\partial)u=0$. It contradicts our assumption. 

If we write
 $P(x,\partial)=\sum_{s=0}^{N}x^{N-s}P_{s}(\partial)\partial^{\max\{m-s,0\}}$,
 the Laplace transform 
\begin{align*}
 P(-\partial,x)&=\sum_{s=0}^{N}(-\partial)^{N-s}P_{s}(x)x^{\max\{m-s,0\}}\\
&=\sum_{s=0}^{N}Q_{s}(x)x^{\max\{m-s,0\}}(-\partial)^{N-s}
\end{align*} 
for $P_{s}(x),Q_{s}(x)\in\C[x]$ and $P_{0}(x)=Q_{0}(x)$. By the assumption $P_{0}(0)=Q_{0}(0)\neq 0$, it follows that $x=0$ is a regular singular point of $P(\partial,x)u=0$.

Let us take a power series $x^{\mu}\sum_{s=0}^{\infty}d_{s}x^{s}$. Then we can see that 
\begin{multline*}
\mathcal{L}P(x,\partial)x^{\mu}\sum_{s=0}^{\infty}d_{s}x^{s}=\sum_{j=0}^{N}(-\partial)^{N-j}P_{j}(x)x^{\mu}\sum_{s=0}^{\infty}d_{s}x^{s}\\
=\sum_{s=m}^{\infty}x^{\mu-N+s}\sum_{k=m}^{s}d_{s-k}\sum_{l=0}^{k}(-\mu-s+l)(-\mu-s+l+1)\cdots(-\mu-s+N-1)\frac{P^{(k-l)}_{l}(0)}{(k-l)!}\\
=\sum_{s=m}^{\infty}x^{\mu-N+s}\sum_{k=m}^{s}d_{s-k}\sum_{l=0}^{k}(-\mu-s+l)_{(k-l)}(-\mu-s+k)_{(N-k)}\frac{P^{(k-l)}_{l}(0)}{(k-l)!}\\
=\sum_{s=m}^{\infty}x^{\mu-N+s}\sum_{k=m}^{s}(-\mu-s+k)_{(N-k)}d_{s-k}p_{k}(-\mu-1,0).
\end{multline*}
Therefore if $\mathcal{L}P(x,\partial)x^{\mu}\sum_{s=0}^{\infty}d_{s}x^{s}=0$, then it must be satisfied that
\[
\sum_{k=m}^{s}d_{s-k}(-\mu-s+k)_{(N-k)}p_{k}(-\mu-1-s,0)
\]
for $s=m,m+1,\ldots$. By the assumption we have
\[p_{k}(-\mu_{i}-j-k,0)=0
\]
 for 
\begin{gather*}i=1,\ldots,l,\\
j=0,\ldots,m_{i}-1,\\
k=m,m+1,\ldots,m+m_{i}-j-1.
\end{gather*}
Also we have
\[
(-i-s+k)_{(N-k)}=0
\] 
for $i=0,\ldots,(N-m)-1$ and $k=m,m+1,\ldots,N-j-1$. Hence the
 differential equation $\mathcal{L}P(x,\partial)u=0$ has a regular singular point at $x=0$ and semi-simple characteristic exponents
\[
\{[0]_{N-m},[\mu_{1}-1]_{m_{1}},\ldots,[\mu_{l}-1]_{m_{l}}\}.
\]
The converse direction can be shown by the same way.
\end{proof}
The same thing can be shown for the Fourier-Laplace inverse transform.
\begin{prop}\label{laplace inverse transform of regular point}
Let us take  $P(x,\partial)\in W[x]$ of $\deg P=N$ and $\mu_{1},\ldots,\mu_{l}\in \C$ such that $\mu_{i}\notin \Z$ and $\mu_{i}-\mu_{j}\notin \Z$ $(i\neq j)$. Then the followings are
 equivalent.
\begin{enumerate}
\item The differential equation $P(x,\partial)u=0$ has semi-simple exponents
\[
 \{[\mu_{1}]_{m_{1}},\ldots,[\mu_{l}]_{m_{l}}\}
\]
at $x=\infty.$
Here $m=\sum_{i=1}^{l}m_{i}.$
\item For the Fourier-Laplace inverse transform $P(\partial,-x)$ of
      $P(x,\partial)$ has 
a regular singular point at $x=0$ and $P(\partial,-x)u=0$ has the semi-simple characteristic exponents,
 \[
 \{[0]_{N-m}[\mu_{1}-1]_{m_{1}},\ldots,[\mu_{l}-1]_{m_{l}}\}
\]
at $x=0$.
\end{enumerate}
\end{prop} 
\begin{proof}
The condition 1 is equivalent to that $P(-x,-\partial)$ has semi-simple exponents
\[
 \{[\mu_{1}]_{m_{1}},\ldots,[\mu_{l}]_{m_{l}}\}
\]
at $x=\infty$. This is equivalent to that $P(\partial,-x)$ has
 semi-simple exponents
 \[
 \{[0]_{N-m}[\mu_{1}-1]_{m_{1}},\ldots,[\mu_{l}-1]_{m_{l}}\}
\]
at $x=0$ by Proposition \ref{oouue}. 
\end{proof}
\begin{cor}\label{laplace transform of rank1}
Let us take  $P(x,\partial)\in W[x]$ of $\deg P=N$ and $\mu_{1},\ldots,\mu_{l}\in \C$ such that $\mu_{i}\notin \Z$ and $\mu_{i}-\mu_{j}\notin \Z$ $(i\neq j)$. Then the followings are
 equivalent.
\begin{enumerate}
\item The differential equation $P(x,\partial)u=0$ has  $e^{\alpha x}$-twisted semi-simple exponents
\[
 \{[\mu_{1}]_{m_{1}},\ldots,[\mu_{l}]_{m_{l}}\}.
\]
\item For the Laplace transform $P(-\partial,x)$ of $P(x,\partial)$ has a
 regular singular point at $x=\alpha$ and  semi-simple exponents
\[
 \{[0]_{N-m}[\mu_{1}-1]_{m_{1}},\ldots,[\mu_{l}-1]_{m_{l}}\}
\]
at $x=\alpha$. Here $m=\sum_{i=1}^{l}m_{i}.$
\end{enumerate}
\end{cor}
\begin{proof}
The condition 1 is equivalent to that
$P(x,\partial+\alpha)u=0$ has semi-simple exponents
\[
  \{[\mu_{1}]_{m_{1}},\ldots,[\mu_{l}]_{m_{l}}\}
\]
at $x=\infty$.
 Hence it is equivalent to that the Laplace
 transform $P(-\partial,x+\alpha)u=0$ has a regular singular point at
 $x=0$, i.e., $\mathcal{L}P(x,\partial)=P(-\partial,x)$ has a regular
 singular point at $x=\alpha$ and semi-simple exponents 
\[
  \{[0]_{N-n}[\mu_{1}-1]_{m_{1}},\ldots,[\mu_{l}-1]_{m_{l}}\}
\] 
at $x=\alpha$.
\end{proof}
For the inverse transform we can show the following as well.
\begin{cor}\label{laplace inverse transform of rank1}
Let us take  $P(x,\partial)\in W[x]$ of $\deg P=N$ and $\mu_{1},\ldots,\mu_{l}\in \C$ such that $\mu_{i}\notin \Z$ and $\mu_{i}-\mu_{j}\notin \Z$ $(i\neq j)$. Then the followings are
 equivalent.
\begin{enumerate}
\item The differential equation $P(x,\partial)u=0$ has  $e^{\alpha x}$-twisted semi-simple exponents
\[
 \{[\mu_{1}]_{m_{1}},\ldots,[\mu_{l}]_{m_{l}}\}.
\]
\item For the Fourier-Laplace inverse transform $P(\partial,-x)$ of $P(x,\partial)$ has
a regular singular point at $x=-\alpha$ and  semi-simple exponents
\[
 \{[0]_{N-m}[\mu_{1}-1]_{m_{1}},\ldots,[\mu_{l}-1]_{m_{l}}\}
\]
at $x=-\alpha$. Here $m=\sum_{i=1}^{l}m_{i}.$
\end{enumerate}
\end{cor}
\begin{proof}
We can show this by the same argument as Corollary \ref{laplace
 transform of rank1}.
\end{proof}
\begin{cor}\label{laplace transform of rank2}
Let us take  $P(x,\partial)\in W[x]$ of $\deg P=N$ and $\mu_{1},\ldots,\mu_{l}\in \C$ such that $\mu_{i}\notin \Z$ and $\mu_{i}-\mu_{j}\notin \Z$ $(i\neq j)$. Then the followings are
 equivalent.
\begin{enumerate}
\item The differential equation $P(x,\partial)u=0$ of  has       $e^{\frac{\alpha}{2}x^{2}+\beta x}$-twisted semi-simple exponents
\[
 \{[\mu_{1}]_{m_{1}},\ldots,[\mu_{l}]_{m_{l}}\}.
\]
\item The Laplace transform $P(-\partial,x)u=0$ has $e^{-\frac{1}{2\alpha}x^{2}+\frac{\beta}{\alpha} x}$-twisted semi-simple exponents
\[
 \{[\mu_{1}]_{m_{1}},\ldots,[\mu_{l}]_{m_{l}}\}.
\]
\end{enumerate}
\end{cor}
\begin{proof}
The condition 1 is equivalent to that $P(x,\partial+\alpha x+\beta)u=0$ has
 semi-simple  exponents
 \[
 \{[\mu_{1}]_{m_{1}},\ldots,[\mu_{l}]_{m_{l}}\}.
\]
at $x=\infty$. On the
 other hand, the condition 2 is equivalent to that $P(-\partial+\frac{1}{\alpha}x-\frac{\beta}{\alpha},x)u=0$ has
 the same semi-simple exponents at infinity. If we
 put $x=\alpha y$, it is equivalent to say that
 $P(y-\frac{1}{\alpha}\partial_{y}-\frac{\beta}{\alpha},\alpha y)v=0$ has
 the same semi-simple exponents at infinity. Here $\partial_{y}=\frac{d}{dy}$. It is easy to see that 
\begin{align*}
\mathcal{L}\circ
 \mathrm{Ade}((\frac{1}{2\alpha}y^{2}-\frac{\beta}{\alpha}y))\circ
 \mathcal{L}^{-1}P(y-\frac{1}{\alpha}\partial_{y}-\frac{\beta}{\alpha},\alpha
 y)=P(y,\partial_{y}+\alpha y+\beta).
\end{align*}
If we notice that for a solution $u$ of $Q(x,\partial)u=0$, $v=e^{p(x)}u$
 is a solution of $\mathrm{Ade}(p(x))Q(x,\partial)v=0$. Since $e^{p(x)}$
 is holomorphic at $x=0$, the multiplication of $e^{p(x)}$ does not
 change exponents at $x=0$. Then the equivalence 1 and 2  follows from Proposition \ref{oouue}. 
\end{proof}
\begin{cor}\label{laplace inverse transform of rank2}
Let us take  $P(x,\partial)\in W[x]$ of $\deg P=N$ and $\mu_{1},\ldots,\mu_{l}\in \C$ such that $\mu_{i}\notin \Z$ and $\mu_{i}-\mu_{j}\notin \Z$ $(i\neq j)$. Then the followings are
 equivalent.
\begin{enumerate}
\item The differential equation $P(x,\partial)u=0$ has       $e^{\frac{\alpha}{2}x^{2}+\beta x}$-twisted semi-simple exponents
\[
 \{[\mu_{1}]_{m_{1}},\ldots,[\mu_{l}]_{m_{l}}\}.
\]
\item The Fourier-Laplace inverse transform $P(\partial,-x)u=0$ has $e^{-\frac{1}{2\alpha}x^{2}-\frac{\beta}{\alpha} x}$-twisted semi-simple exponents
\[
 \{[\mu_{1}]_{m_{1}},\ldots,[\mu_{l}]_{m_{l}}\}.
\]
\end{enumerate}
\end{cor}
\begin{proof}
The condition 2 is equivalent to that
 $P(\partial-\frac{1}{\alpha}x-\frac{\beta}{\alpha},-x)$ has the above
 semi-simple exponents at $x=\infty$. If we put $y=-\frac{1}{\alpha}x$,
 this is equivalent to
 $P(-\frac{1}{\alpha}\partial_{y}+y-\frac{\beta}{\alpha},\alpha y)$ has
 the same exponents at $y=\infty$. Also  we have
\[
 \mathcal{L}^{-1}\circ\mathrm{Ade}(\frac{1}{2\alpha}y^{2}+\frac{\beta}{\alpha})\circ
 \mathcal{L}P(-\frac{1}{\alpha}\partial_{y}+y-\frac{\beta}{\alpha},\alpha
 y)=P(y,\partial_{y}+\alpha y+\beta).
\]
As in Corollary \ref{laplace transform of rank2}, we can show this corollary.
\end{proof}
\begin{prop}\label{inverse of euler transform}
Let us take $P(x,\partial)\in W[x]$. We assume that $P(x,\partial)$ can be written by
\[
\mathrm{R}P(x,\partial)=x^{N}\prod_{i=0}^{r}(\partial-\alpha_{i})^{m_{i}}+\sum_{j=1}^{N-1}x^{N-j}\prod_{i}^{r}(\partial-\alpha_{i})^{\max{\{m_{i}-j,0\}}}P_{j}(\partial)
\]
for $P_{i}(x)\in \C[x]$
and 
$P(x,\partial)$ has $e^{\alpha_{i}x}$-twisted semi-simple exponents
\[
\{[\mu^{i}_{1}]_{m^{i}_{1}},\ldots,[\mu^{i}_{l_{i}}]_{m^{i}_{l_{i}}}\}
\]
where $m_{i}=\sum_{j=1}^{l_{i}}m_{j}^{i}$. Moreover we assume that $\mu^{i}_{j}\notin \Z$ for $i=1,\ldots,r$ and $j=1,\ldots,l_{i}$ and $\mu^{i}_{j}-\mu^{i}_{k}\notin\Z$ $(k\neq j)$.

Then  we have the followings.
\begin{enumerate}
\item We have
\[
E(*;0)P(x,\partial)\sim P(x,\partial)
\]
for $i=1,\ldots,r$.
\item For fixed $i\in\{1,\ldots,r\}$ and  $\mu\in\C$ such that
      $\mu\notin \Z$ and $\mu^{i}_{j}-\mu\notin \Z\backslash\{1\}$ for all
      $j=1,\ldots l_{i}$,  we have
\[
E(\alpha_{i};-\mu)E(\alpha_{i};\mu)P(x,\partial)\sim P(x,\partial)
\]
for $i=1,\ldots,r$.
\end{enumerate}
\end{prop}
\begin{proof}
First we  show 1. We have
\begin{align*}
E(\alpha_{i};0)P(x,\partial)&=\mathcal{L}\mathrm{R}\mathcal{L}^{-1}\mathrm{R}P(x,\partial)\\
&=\mathcal{L}\mathrm{R}(\prod_{i=1}^{r}(-x-\alpha_{i})^{m_{i}}\partial^{N}\\&+\sum_{j=1}^{N}\prod_{i=1}^{r}(-x-\alpha_{i})^{\max\{m_{i}-j,0\}}a_{j}(x)\partial^{N-j}).
\end{align*}
By the assumption and Corollary \ref{laplace inverse transform of rank1}, every
 characteristic exponent at the regular singular point $x=-\alpha_{i}$ of
 $\mathcal{L}^{-1}\mathrm{R}P(x,\partial)$ is not integers. Hence by Proposition \ref{reduce degree}, $\mathrm{R}\mathcal{L}^{-1}\mathrm{R}P(x,\partial)= \mathcal{L}^{-1}\mathrm{R}P(x,\partial)$. Hence $E(*;0)P(x,\partial)=\mathcal{L}\mathrm{R}\mathcal{L}^{-1}\mathrm{R}P(x,\partial)\sim P(x,\partial).$

Let us show 2. Fix an $i\in\{1,\ldots r\}$. If $\mu^{i}_{j}-\mu\neq 1$ for all $j=1,\ldots,l_{i}$,
 $E(\alpha_{i},\mu)P(x,\partial)$ has $e^{\alpha_{i}x}$-twisted semi-simple exponents
\[
\{[-\mu+1]_{N-m_{i}}[\mu^{i}_{1}-\mu]_{m^{i}_{1}},\ldots,[\mu^{i}_{l_{i}}-\mu]_{m^{i}_{l_{i}}}\}
\]
and $e^{\alpha_{i'}x}$-twisted semi-simple exponents
\[
\{[\mu^{i'}_{1}]_{m^{i'}_{1}},\ldots,[\mu^{i'}_{l_{i'}}]_{m^{i'}_{l_{i'}}}\}
\]
for the other $i'$ by corollaries from $\ref{laplace transform of
 rank1}$ to $\ref{laplace inverse transform of rank2}$. On the other
 hand, if there exist $j\in\{1,\ldots,l_{i}\}$ such that
 $\mu^{i}_{j}-1=u$, then $E(\alpha_{i},\mu)P(x,\partial)$ has $e^{\alpha_{i}x}$-twisted semi-simple exponents
\[
\{[\mu^{i}_{1}-\mu]_{m^{i}_{1}},\ldots,[\mu^{i}_{j-1}-\mu]_{m^{i}_{j-1}},[-\mu+1]_{N-m_{i}},[\mu^{i}_{j+1}-\mu]_{m^{i}_{j+1}},\ldots,[\mu^{i}_{l_{i}}-\mu]_{m^{i}_{l_{i}}}\}
\]
and $e^{\alpha_{i'}x}$-twisted semi-simple exponents
\[
\{[\mu^{i'}_{1}]_{m^{i'}_{1}},\ldots,[\mu^{i'}_{l_{i'}}]_{m^{i'}_{l_{i'}}}\}
\]
for the other $i'$ by corollaries from $\ref{laplace transform of
 rank1}$ to $\ref{laplace inverse transform of rank2}$ In  both cases,
 characteristic exponents are not integers and moreover difference of
 them are not integers. Hence we can see 
\[
 E(*,0)E(\alpha_{i},\mu)P\sim E(\alpha_{i},\mu)P
\]
by the same argument as 1. Hence we have 2 as well as the proof of
 Proposition \ref{mcirred}. 

\end{proof}

\subsection{Twisted Euler transform}

\begin{df}[Normal at infinity]
Let us take $P(x,\partial)\in W[x,\xi]$ which has an irregular singular point at $x=\infty$ of rank 2. The order of $P(x,\partial)$ is $n$. 

We say that $P(x,\partial)$ is normal at the irregular singular point $x=\infty$, if the followings are satisfied. There exist $\alpha_{i},\beta^{i}_{j}\in\C$ for $i=1,\ldots,r$ and $j=1,.\ldots,l_{i}$.  
\begin{enumerate}
\item For every $i=1,\ldots,r$, the Newton polygon of $P(x,\partial+\alpha_{i}x)$ has only three
      vertices
      $s^{i}_{0}=(u^{i}_{0}=0,v^{i}_{0}),s^{i}_{1}=(u^{i}_{1},v^{i}_{1})$
      and $s^{i}_{3}=(u^{i}_{3}=n,v^{i}_{3})$. Corresponding slopes are
      $1$ and $2$. Length of the segment $[s^{i}_{1},s^{i}_{0}]$ is $n_{i}$. Here $n=\sum_{i=1}^{r}n_{i}$.
\item For every $(i,j)$, $i=1,\ldots,r$ and $j=1,\ldots,l_{i}$, the Newton polygon of $P(x,\partial+\alpha_{i}x+\beta^{i}_{j})$ has only four vertices $s^{ij}_{0}=(u^{ij}_{0}=0,v^{ij}_{0}),s^{ij}_{1}=(u^{ij}_{1},v^{ij}_{1}),s^{ij}_{2}=(u^{ij}_{2},v^{ij}_{2}),s^{ij}_{3}=(u^{ij}_{3}=n,v^{ij}_{3})$. Corresponding slopes are $0$, $1$ and $2$. Lengths of segments $[s^{ij}_{1},s^{ij}_{0}]$, $[s^{ij}_{2},s^{ij}_{1}]$ and $[s^{ij}_{3},s^{ij}_{2}]$ are $n^{i}_{j}$, $n_{i}-n^{i}_{j}$ and $n-n_{i}$ respectively. Here $n_{i}=\sum_{j=1}^{l_{i}}n^{i}_{j}$.
\end{enumerate}
\end{df}
\begin{rem}
By Proposition \ref{formal regular solutions}, Proposition \ref{frobenious method}
and Remark \ref{Newton polygon and formal solutions}, $P(x,\partial)\in
W[x,\xi]$ is normal at infinity if and only if there are
$n^{i}_{j}$-dimensional space of 
$e^{\frac{\alpha_{i}}{2}x^{2}+\beta^{i}_{j}x}$-twisted formal
solutions. Since $n=\sum_{i=1}^{r}\sum_{j=1}^{l_{i}}n^{i}_{j}$ coincides
with the order of $P(x,\partial)$, these are all of formal solutions of
$P(x,\partial)u=0$ at $x=\infty$.  \end{rem}

Let us take $P(x,\partial)\in W[x,\xi]$ which has regular singular points
 at $c_{1},\ldots,c_{p}\in\C$, the irregular singular point of rank 2 at $x=\infty$ and
 no other singular points. Then we can write
\begin{equation}\label{standard form}
 P(x,\partial)=\sum_{i=0}^{n}\prod_{j=1}^{p}(x-c_{j})^{n-i}a_{i}(x)\partial^{n-i},
\end{equation}
where $a_{i}(x)\in \C[x,\xi]$ for $i=1,\ldots,n$ and $a_{0}(x)=1$. Since
 the rank of irregularity is 2, degrees of $a_{i}(x)$ have upper bound,
\[
 \deg a_{i}(x)\le (p+1)i.
\] 
We call $(\ref{standard form})$ the standard form of $P(x,\partial)$.

\begin{df}[Table of local datum]\label{local datum}
For sufficiently large numbers $K$ and $K'$, we take  $\xi=\{\mu_{j}^{i}\mid 1\le i,j\le K\}\cup
 \{\nu^{ij}_{k}\mid 1\le i,j,k\le K'\}$ as the indeterminants of $\C(\xi)$.
Let us take $P(x,\partial)\in W[x,\xi]$ which has regular singular points
 at $c_{1},\ldots,c_{p}\in\C$, the irregular singular point of rank 2 at $x=\infty$ and
 no other singular points. Moreover $P(x,\partial)$ is normal at infinity. 
 
Let us assume that the differential equation $P(x,\partial)u=0$ has the
 following local solutions.
 
\begin{itemize}
\item Around each regular singular point $x=c_{i}$, it has semi-simple exponents
\[
\{[0]_{m_{0}^{i}}[\mu_{1}^{i}]_{m_{1}^{i}},\ldots,[\mu_{s_{i}}^{i}]_{m_{s_{i}}^{i}}\}
\] 
where $m_{j}^{i}\in\Z_{>0}$ for $i=1,\ldots,s_{i}$ and $m_{0}^{i}\in
      \Z_{\ge 0}$ which satisfy
	     $\sum_{j=0}^{s_{i}}m_{j}^{i}=n=\mathrm{ord}\,P$ for $i=1,\ldots,p$. For the
      simplicity, we put 
\begin{align*}
 &\mu^{i}=(\mu_{1}^{i},\ldots,\mu_{s_{i}}^{i}), &m^{i}=(m_{1}^{i},\ldots,m_{s_{i}}^{i}),
\end{align*}
and write
\[
 [\mu^i;m^{i}]=\{[\mu_{1}^{i}]_{m_{1}^{i}},\ldots,[\mu_{s_{i}}^{i}]_{m_{s_{i}}^{i}}\},
\]
shortly.
\item Around irregular singular point $x=\infty$, it has
      $e^{\frac{\alpha_{i}}{2}x^{2}+\beta_{j}x }$-twisted semi-simple exponents
\[
 \{[\nu^{ij}_{1}]_{n^{ij}_{1}},\ldots,[\nu^{ij}_{t_{ij}}]_{n^{ij}_{t_{ij}}}\}
\] 
for $i=1,\ldots,r$, $j=1,\ldots,l_{i}$. Here $n_{k}^{ij}\in \Z_{>0}$ and $\sum_{i=1}^{r}\sum_{j=1}^{l_{i}}\sum_{k=1}^{t_{ij}}n^{ij}_{k}=n=\mathrm{ord}\,P(x,\partial)$. 
We put $n^{i}_{j}=\sum_{k=1}^{t_{ij}}n^{ij}_{k}$ and  $n_{i}=\sum_{j=1}^{l_{i}}n^{i}_{j}$. For the simplicity, we put
\begin{align*}
&\nu^{ij}=(\nu^{ij}_{1},\ldots,\nu^{ij}_{t_{ij}}), &n^{ij}=(n^{ij}_{1},\ldots,n^{ij}_{t_{ij}}),
\end{align*}
and write
\[
 [\nu^{ij};n^{ij}]= \{[\nu^{ij}_{1}]_{n^{ij}_{1}},\ldots,[\nu^{ij}_{t_{ij}}]_{n^{ij}_{t_{ij}}}\}.
\]
\end{itemize}
Then we write the following table of local exponents of $P(x,\partial)$,\\
\begin{center}
\begin{tabular}{|c|c||c|c|c|c|c|c|c|}
\hline
\multicolumn{2}{|c||}{}&\multicolumn{2}{|c|}{$\alpha_{1}$}&\multicolumn{2}{c|}{$\alpha_{2}$}&$\cdots$&\multicolumn{2}{c|}{$\alpha_{r}$}\\
\hline
$c_{1}$&$[\mu^{1};m^{1}]$&$\beta^{1}_{1}$&$[\nu^{11};n^{11}]$&$\beta^{2}_{1}$&$[\nu^{21};n^{21}]$&$\cdots$&$\beta^{r}_{1}$&$[\nu^{r1};n^{r1}]$\\

$c_{2}$&$[\mu^{2};m^{2}]$&$\beta^{1}_{2}$&$[\nu^{12};n^{12}]$&$\beta^{2}_{2}$&$[\nu^{22};n^{22}]$&$\cdots$&$\beta^{r}_{2}$&$[\nu^{r2};n^{r2}]$\\
$\vdots$&$\vdots$&$\vdots$&$\vdots$&$\vdots$&$\vdots$&&$\vdots$&$\vdots$\\
 $c_{p}$&$[\mu^{p};m^{p}]$&$\beta^{1}_{l_{1}}$&$[\nu^{1l_{1}};n^{1l_{1}}]$&$\beta^{2}_{l_{2}}$&$[\nu^{2l_{2}};n^{2l_{2}}]$&&$\beta^{r}_{l_{r}}$&$[\nu^{rl_{r}};n^{rl_{r}}]$\\
\hline
\end{tabular}.
\end{center}
We call this table the table of local datum of $P(x,\partial)$. 

In particular, if $P(x,\partial)=\partial+\alpha x+\beta$ for some $\alpha,\beta\in\C$, we say $P(x,\partial)$ has the trivial table of local datum.
\end{df}

\begin{rem}
In Definition \ref{local datum},  the indeterminants $\xi=\{\mu_{j}^{i}\mid 1\le i,j\le K\}\cup
 \{\nu^{ij}_{k}\mid 1\le i,j,k\le K'\}$ have only one linear relation
 which comes from Fuchs relation (see \cite{Bert} and \cite{B-L} for example).

\end{rem}
\begin{thm}\label{lapalce transform of riemann scheme}
Let us take $P(x,\partial)\in W[x,\xi]$ as in Definition \ref{local datum}. We assume that $P(x,\partial)$ has the following nontrivial table of local datum,
\begin{center}
\begin{tabular}{|c|c||c|c|c|c|c|c|c|}
\hline
\multicolumn{2}{|c||}{}&\multicolumn{2}{|c|}{$\alpha_{1}$}&\multicolumn{2}{c|}{$\alpha_{2}$}&$\cdots$&\multicolumn{2}{c|}{$\alpha_{r}$}\\
\hline
$c_{1}$&$[\mu^{1};m^{1}]$&$\beta^{1}_{1}$&$[\nu^{11};n^{11}]$&$\beta^{2}_{1}$&$[\nu^{21};n^{21}]$&$\cdots$&$\beta^{r}_{1}$&$[\nu^{r1};n^{r1}]$\\

$c_{2}$&$[\mu^{2};m^{2}]$&$\beta^{1}_{2}$&$[\nu^{12};n^{12}]$&$\beta^{2}_{2}$&$[\nu^{22};n^{22}]$&$\cdots$&$\beta^{r}_{2}$&$[\nu^{r2};n^{r2}]$\\
$\vdots$&$\vdots$&$\vdots$&$\vdots$&$\vdots$&$\vdots$&&$\vdots$&$\vdots$\\
 $c_{p}$&$[\mu^{p};m^{p}]$&$\beta^{1}_{l_{1}}$&$[\nu^{1l_{1}};n^{1l_{1}}]$&$\beta^{2}_{l_{2}}$&$[\nu^{2l_{2}};n^{2l_{2}}]$&&$\beta^{r}_{l_{r}}$&$[\nu^{rl_{r}};n^{rl_{r}}]$\\
\hline
\end{tabular}.
\end{center}Changing the order of $\alpha_{1},\ldots,\alpha_{r}$ and applying $Ade(\frac{\alpha_{1}}{2}x^{2})$, We can assume that $\alpha_{1}=0$. Then $\mathcal{L}\circ \mathrm{R}\,P(x,\partial)$ has 
\[
\mathrm{ord}\,\mathcal{L}\circ \mathrm{R}P(x,\partial)=\sum_{i=1}^{p}\sum_{j=1}^{s_{i}}m^{i}_{j}+(n-n_{1}),
\]
and the table of local datum,
\begin{center}
\begin{tabular}{|c|c||c|c|c|c|c|c|c|}
\hline
\multicolumn{2}{|c||}{}&\multicolumn{2}{|c|}{$0$}&\multicolumn{2}{c|}{$-\frac{1}{\alpha_{2}}$}&$\cdots$&\multicolumn{2}{c|}{$-\frac{1}{\alpha_{r}}$}\\
\hline
$\beta^{1}_{1}$&$[\tilde{\nu}^{11};n^{11}]$&$-c_{1}$&$[\tilde{\mu}^{1};m^{1}]$&$\frac{\beta^{2}_{1}}{\alpha_{2}}$&$[\nu^{21};n^{21}]$&$\cdots$&$\frac{\beta^{r}_{1}}{\alpha_{r}}$&$[\nu^{r1};n^{r1}]$\\

$\beta^{1}_{2}$&$[\tilde{\nu}^{12};n^{12}]$&$-c_{2}$&$[\tilde{\mu}^{2};m^{2}]$&$\frac{\beta^{2}_{2}}{\alpha_{2}}$&$[\nu^{22};n^{22}]$&$\cdots$&$\frac{\beta^{r}_{2}}{\alpha_{r}}$&$[\nu^{r2};n^{r2}]$\\
$\vdots$&$\vdots$&$\vdots$&$\vdots$&$\vdots$&$\vdots$&&$\vdots$&$\vdots$\\
$\beta^{1}_{l_{1}}$&$[\tilde{\nu}^{1l_{1}};n^{1l_{1}}]$& $-c_{p}$&$[\tilde{\mu}^{p};m^{p}]$&$\frac{\beta^{2}_{l_{2}}}{\alpha_{2}}$&$[\nu^{2l_{2}};n^{2l_{2}}]$&&$\frac{\beta^{r}_{l_{r}}}{\alpha_{r}}$&$[\nu^{rl_{r}};n^{rl_{r}}]$\\
\hline
\end{tabular}.
\end{center}
Here 
\begin{align*}
& \tilde{\mu}^{i}=(\mu^{i}_{1}+1,\mu^{i}_{2}+1,\ldots,\mu^{i}_{s_{i}}+1), &\tilde{\nu^{1j}}=(\nu^{1j}_{1}-1,\nu^{1j}_{2}-1,\ldots,\nu^{1j}_{t_{1j}}-1),
\end{align*}
for $i=1,\ldots,p$ and $j=1,\ldots,l_{1}$.
\end{thm}
\begin{proof}
By Proposition \ref{strict regular}, the standard form of $P(x,\partial)$ can
 be divided by $\phi(x)=\prod_{i=1}^{p}(x-c_{i})^{m_{i}}$, i.e., there
 exist $Q(x,\partial)\in W[x,\xi]$ such that $P(x,\partial)=\phi(x)Q(x,\partial)$
 and moreover $\mathrm{R}P(x,\partial)=Q(x,\partial)$. Since we assume
 that $\alpha_{1}=0$, we can see $\deg
 P(x,\partial)=(p+1)(n-n_{1})$. Hence $\deg Q(x,\partial)=\deg
 P(x,\partial)-\sum_{i=1}^{p}m_{i}=(p+1)(n-n_{1})-\sum_{i=1}^{p}m_{i}=\sum_{i=1}^{p}(n-m^{i}_{0})+(n-n_{1})=\sum_{i=1}^{p}\sum_{j=1}^{s_{i}}m^{i}_{j}+(n-n_{1})$.

Then this theorem is obtained by Corollary \ref{laplace transform
 of rank1} and Corollary  \ref{laplace transform of rank2}.
\end{proof}
\begin{rem}
We can show the same thing as the above theorem for the Fourier-Laplace
 inverse transform by Corollary
 \ref{laplace inverse transform of rank1} and Corollary \ref{laplace
 inverse transform of rank2}.
\end{rem}
\begin{df}[The rigidity index]
Let us take $P(x,\partial)\in W[x,\xi]$ as in Definition \ref{local datum}. Then we define the number
\begin{multline*}
\mathrm{idx}\,P=\\
-((p+1)n^{2}-\sum_{i=1}^{r}n_{i}^{2}-\sum_{i=1}^{r}\sum_{j=1}^{l_{i}}(n^{i}_{j})^{2}-\sum_{i=1}^{p}\sum_{j=0}^{s_{i}}(m^{i}_{j})^{2}-\sum_{i=1}^{q}\sum_{j=1}^{l_{i}}\sum_{k=1}^{t_{ij}}(n^{ij}_{k})^{2}),
\end{multline*}
and call this the rigidity index of $P$.
\end{df}
\begin{rem}
The standard form of $P(x,\partial)$ is
\begin{equation*}
 P(x,\partial)=\sum_{i=0}^{n}\prod_{j=1}^{p}(x-c_{j})^{n-i}a_{i}(x)\partial^{n-i},
\end{equation*}
where 
\[
 \deg a_{i}(x)\le (p+1)i.
\] 
Hence $P(x,\partial)$ has 
\[
\sum_{i=1}^{n}((p+1)i+1)
\]
coefficients in $\overline{\C(\xi)}$. The informations about the sets of exponents at regular singular points $x=c_{1},\ldots, c_{p}$ require
\[
\sum_{i=1}^{p}\sum_{j=0}^{s_{i}}\frac{m^{i}_{j}(m^{i}_{j}+1)}{2}
\]
linear equations by Lemma \ref{strict regular}.  On the other hand,
 $P(x,\partial)$ has $n_{i}$-dimensional formal solutions with
 $e^{\frac{\alpha_{i}}{2}x^{2}+\beta^{i}_{j}x}$-twisted semi-simple exponents for $i=1,\ldots,r$ and $j=1,\ldots,l_{i}$. Hence the standard form of $P(x,\partial+\alpha_{i} x)$ is
\begin{equation*}
 P(x,\partial+\alpha_{i} x)=\sum_{k=0}^{n}\prod_{l=1}^{p}(x-c_{l})^{n-k}b_{k}^{i}(x)\partial^{n-k},
\end{equation*}
where \[
\begin{cases}
 \deg b^{i}_{k}(x)\le (p+1)k &\text{for}\ 1\le k\le (n-n_{i})\\
 \deg b^{i}_{(n-n_{i})+m}(x)\le pm +(p+1)(n-n_{i})&\text{for}\ 1\le m\le n_{i}
\end{cases}.\]
Hence this requires
\[
\sum_{k=1}^{n}(\deg a_{k}(x)-\deg b^{i}_{k}(x))=\frac{n_{i}(n_{i}+1)}{2}
\]
linear equations of $\overline{\C(\xi)}$. Moreover $P(x,\partial+\alpha_{i} x+\beta^{i}_{j})$ has the standard form
\begin{equation*}
 P(x,\partial+\alpha_{i} x+\beta^{i}_{j})=\sum_{k=0}^{n}\prod_{l=1}^{p}(x-c_{l})^{n-k}c_{k}^{ij}(x)\partial^{n-k},
\end{equation*}
where \[
\begin{cases}
 \deg c^{ij}_{k}(x)\le (p+1)k &\text{for}\ 1\le k\le (n-n_{i})\\
 \deg c^{ij}_{(n-n_{i})+l}(x)\le pl +(p+1)(n-n_{i})&\text{for}\ 1\le l\le (n_{i}-n^{i}_{j})\\
 \deg c^{ij}_{(n-(n_{i}-n^{i}_{j}))+m}\le (p-1)m+p(n_{i}-n^{i}_{j})+(p+1)(n-n_{i})&\text{for}\ 1\le m\le n^{i}_{j}
\end{cases}.\]
Hence this requires
\[
\sum_{k=1}^{n}(\deg b^{i}_{k}(x)-\deg c^{ij}_{k}(x))=\frac{n^{i}_{j}(n^{i}_{j}+1)}{2}
\]
linear equations of $\overline{\C(\xi)}$. Finally, informations about the sets of exponents
require
\[
\sum_{k=1}^{t_{ij}}\frac{n^{ij}_{k}(n^{ij}_{k}+1)}{2}
\]
linear equations of $\overline{\C(\xi)}$ by Proposition \ref{strictly
 normal}. There is one more linear equation from Fuchs
 relation. Hence there are at most the following parameters in $P(x,\partial)$,
\begin{multline*}
\sum_{i=1}^{n}((p+1)i+1)-\sum_{i=1}^{p}\sum_{j=0}^{s_{i}}\frac{m^{i}_{j}(m^{i}_{j}+1)}{2}-\sum_{i=1}^{r}\frac{n_{i}(n_{i}+1)}{2}\\
-\sum_{i=1}^{r}\sum_{j=1}^{l_{i}}\frac{n^{i}_{j}(n^{i}_{j}+1)}{2}-\sum_{i=1}^{p}\sum_{j=1}^{l_{i}}\sum_{k=1}^{t_{ij}}\frac{n^{ij}_{k}(n^{ij}_{k}+1)}{2}+1=1-\frac{1}{2}\mathrm{idx}\,P.
\end{multline*}
\end{rem}
\begin{thm}\label{addition and middle convolution}
Let us take $P(x,\partial)\in W[x,\xi]$ as in Definition \ref{local datum} with the table of local datum,
\begin{center}
\begin{tabular}{|c|c||c|c|c|c|c|c|c|}
\hline
\multicolumn{2}{|c||}{}&\multicolumn{2}{|c|}{$\alpha_{1}$}&\multicolumn{2}{c|}{$\alpha_{2}$}&$\cdots$&\multicolumn{2}{c|}{$\alpha_{r}$}\\
\hline
$c_{1}$&$[\mu^{1};m^{1}]$&$\beta^{1}_{1}$&$[\nu^{11};n^{11}]$&$\beta^{2}_{1}$&$[\nu^{21};n^{21}]$&$\cdots$&$\beta^{r}_{1}$&$[\nu^{r1};n^{r1}]$\\

$c_{2}$&$[\mu^{2};m^{2}]$&$\beta^{1}_{2}$&$[\nu^{12};n^{12}]$&$\beta^{2}_{2}$&$[\nu^{22};n^{22}]$&$\cdots$&$\beta^{r}_{2}$&$[\nu^{r2};n^{r2}]$\\
$\vdots$&$\vdots$&$\vdots$&$\vdots$&$\vdots$&$\vdots$&&$\vdots$&$\vdots$\\
 $c_{p}$&$[\mu^{p};m^{p}]$&$\beta^{1}_{l_{1}}$&$[\nu^{1l_{1}};n^{1l_{1}}]$&$\beta^{2}_{l_{2}}$&$[\nu^{2l_{2}};n^{2l_{2}}]$&&$\beta^{r}_{l_{r}}$&$[\nu^{rl_{r}};n^{rl_{r}}]$\\
\hline
\end{tabular}.
\end{center}
\begin{enumerate}
\item For $f(\xi)\in \C(\xi)$ and $i=1,\ldots,p$, the table of local datum of $\mathrm{RAd}((x-c_{i})^{f(\xi)})P(x,\partial)$ is
\begin{center}
\begin{tabular}{|c|c||c|c|c|c|c|}
\hline
\multicolumn{2}{|c||}{}&\multicolumn{2}{|c|}{$\alpha_{1}$}&$\cdots$&\multicolumn{2}{c|}{$\alpha_{r}$}\\
\hline
$c_{1}$&$[\mu^{1};m^{1}]$&$\beta^{1}_{1}$&$[\nu^{11}-f(\xi);n^{11}]$&$\cdots$&$\beta^{r}_{1}$&$[\nu^{r1}-f(\xi);n^{r1}]$\\

$c_{2}$&$[\mu^{2};m^{2}]$&$\beta^{1}_{2}$&$[\nu^{12}-f(\xi);n^{12}]$&$\cdots$&$\beta^{r}_{2}$&$[\nu^{r2}-f(\xi);n^{r2}]$\\
$\vdots$&$\vdots$&$\vdots$&$\vdots$&&$\vdots$&$\vdots$\\
$c_{i}$&$[\mu^{i}+f(\xi);m^{i}]$&$\vdots$&$\vdots$&&$\vdots$&$\vdots$\\
$\vdots$&$\vdots$&$\vdots$&$\vdots$&&$\vdots$&$\vdots$\\
 $c_{p}$&$[\mu^{p};m^{p}]$&$\beta^{1}_{l_{1}}$&$[\nu^{1l_{1}}-f(\xi);n^{1l_{1}}]$&&$\beta^{r}_{l_{r}}$&$[\nu^{rl_{r}}-f(\xi);n^{rl_{r}}]$\\
\hline
\end{tabular}.
\end{center}
where $\mu^{i}+f(\xi)=(\mu^{i}_{1}+f(\xi),\ldots,\mu^{i}_{s_{i}}+f(\xi))$ and $\nu^{ij}-f(\xi)=(\nu^{ij}_{1}-f(\xi),\ldots,\nu^{ij}_{t_{ij}}-f(\xi))$.
\item For $i=1,\ldots,r$, $j=1,\ldots,l_{i}$ and $k=1,\ldots,t_{ij}$, the table of local datum of $E(\frac{\alpha_{i}}{2}x^{2}+\beta^{i}_{j}x;\beta^{i}_{j},\nu^{ij}_{k}-1)P(x,\partial)$ is 
\begin{center}
\begin{tabular}{|c|c||c|c|c|c|c|c|c|}
\hline
\multicolumn{2}{|c||}{}&\multicolumn{2}{|c|}{$\alpha_{1}$}&\multicolumn{2}{c|}{$\alpha_{2}$}&$\cdots$&\multicolumn{2}{c|}{$\alpha_{r}$}\\
\hline
$c_{1}$&$[\tilde{\mu}^{1};m^{1}]$&$\beta^{1}_{1}$&$[\tilde{\nu}^{11};\tilde{n}^{11}]$&$\beta^{2}_{1}$&$[\tilde{\nu}^{21};\tilde{n}^{21}]$&$\cdots$&$\beta^{r}_{1}$&$[\tilde{\nu}^{r1};\tilde{n}^{r1}]$\\

$c_{2}$&$[\tilde{\mu}^{2};m^{2}]$&$\beta^{1}_{2}$&$[\tilde{\nu}^{12};\tilde{n}^{12}]$&$\beta^{2}_{2}$&$[\tilde{\nu}^{22};\tilde{n}^{22}]$&$\cdots$&$\beta^{r}_{2}$&$[\tilde{\nu}^{r2};\tilde{n}^{r2}]$\\
$\vdots$&$\vdots$&$\vdots$&$\vdots$&$\vdots$&$\vdots$&&$\vdots$&$\vdots$\\
 $c_{p}$&$[\tilde{\mu}^{p};m^{p}]$&$\beta^{1}_{l_{1}}$&$[\tilde{\nu}^{1l_{1}};\tilde{n}^{1l_{1}}]$&$\beta^{2}_{l_{2}}$&$[\tilde{\nu}^{2l_{2}};\tilde{n}^{2l_{2}}]$&&$\beta^{r}_{l_{r}}$&$[\tilde{\nu}^{rl_{r}};\tilde{n}^{rl_{r}}]$\\
\hline
\end{tabular},
\end{center}
where 
\begin{multline*}
\tilde{\nu}^{xy}=\\
\begin{cases}
\begin{split}(\nu^{ij}_{1}-(\nu^{ij}_{k}-1),\ldots,\nu^{ij}_{k-1}-(\nu^{ij}_{k}-1),-\nu^{ij}_{k},\nu^{ij}_{k+1}-(\nu^{ij}_{k}-1),\ldots)\\
\text{if }(x,y)=(i,j),\end{split}\\
\nu^{iy}\text{ if }x=i \text{ and }y\neq j,\\
(\nu^{xy}_{1}+\nu^{ij}_{k}-1,\ldots,\nu^{xy}_{t_{xy}}+\nu^{ij}_{k}-1)\text{ otherwise},
\end{cases}
\end{multline*}
\[
\tilde{\mu}^{x}=(\mu^{x}_{1}+\nu^{ij}_{k}-1,\ldots,\mu^{x}_{s_{x}}+\nu^{ij}_{k}-1),
\]
\[
\tilde{n}^{xy}=
\begin{cases}
(n^{ij}_{1},\ldots,n^{ij}_{k-1},N_{i}-n^{i}_{j},n^{ij}_{k+1},\ldots)&\text{if } (x,y)=(i,j),\\
n^{xy}&\text{otherwise},
\end{cases}
\]
for $N_{i}=\sum_{k=1}^{p}\sum_{l=1}^{s_{k}}m^{k}_{l}+(n-n_{i})$. The order of $E(\frac{\alpha_{i}}{2}x^{2}+\beta^{i}_{j}x;\beta^{i}_{j},\nu^{ij}_{k}-1)P(x,\partial)$ is $n-n^{ij}_{k}+N_{i}-n^{i}_{j}$.
\end{enumerate}
\end{thm}
\begin{proof}
The first assertion is the  direct consequence of Proposition \ref{addition and exponents} and Proposition \ref{addition and exponents at infinity}. The second assertion follows from Theorem \ref{lapalce transform of riemann scheme} and the first assertion.
\end{proof}
\begin{rem}\label{well-defined middle convolution}
In the second assertion of Theorem \ref{addition and middle convolution}, we see that the order of $E(\frac{\alpha_{i}}{2}x^{2}+\beta^{i}_{j}x;\beta^{i}_{j},\nu^{ij}_{k}-1)P(x,\partial)$ is $(n-n_{k}^{ij})+(N-n^{i}_{j})$.
By the same argument as the proof of Proposition \ref{oouue}, we can see
 that $N_{i}-n^{i}_{j}\ge 0$ for all $i$ and $j$. Hence $E(\frac{\alpha_{i}}{2}x^{2}+\beta^{i}_{j}x;\beta^{i}_{j},\nu^{ij}_{k}-1)P(x,\partial)$ is well-defined, i.e., $(n-n_{k}^{ij})+(N_{i}-n^{i}_{j})>0$, if and only if $n-n^{ij}_{k}>0$ or $N_{i}-n^{i}_{j}>0$.  
On the contrary, if we assume $n=n^{ij}_{k}$ and $N_{i}=n^{i}_{j}$, we can see the following. 
\end{rem}
\begin{lem}\label{reducibility}
Let us take $P(x,\partial)\in W[x,\xi]$ whose table of local datum is
\begin{center}
\begin{tabular}{|c|c||c|c|}
\hline
\multicolumn{2}{|c||}{}&\multicolumn{2}{|c|}{$\alpha$}\\
\hline
$c_{1}$&$[\mu^{1};m^{1}]$&$\beta$&$[(0);(n)]$\\
$c_{2}$&$[\mu^{2};m^{2}]$&&\\
$\vdots$&$\vdots$&&\\
 $c_{p}$&$[\mu^{p};m^{p}]$&&\\
\hline
\end{tabular}. 
\end{center} 
And we assume $n=\sum_{i=1}^{p}\sum_{j=1}^{s_{i}}m^{i}_{j}$.
Then 
\[
P(x,\partial)\sim (\partial-\alpha x-\beta)^{n}.
\]
\end{lem}
\begin{proof}
We can see that $x=\infty$ is a regular singular point of  $P(x,\partial+\alpha x+\beta)$. Hence we can write 
\[
P(x,\partial+\alpha x+\beta)=\sum_{i=0}^{n}x^{n-i}P_{i}(\partial),
\]
where $\deg P_{i}(x)\le n$ for $i=0,\ldots,n$ and $P_{i}^{(j)}(0)=0$ for $i+j<n$.
We consider polynomials
\[
p_{k}(\lambda,x)=\sum_{i=0}^{k}(\lambda+i+1)_{k-i}\frac{P^{(k-i)}_{i}(x)}{(k-i)!}.
\]
The characteristic exponents at $x=\infty$ implies that 
\[
p_{n+j}(-i-j,0)=0
\]
for $i=0,1,\ldots,n-1$ and $j=0,\ldots,n-i-1.$ This implies that $P_{i}^{(j)}(0)=0$ for $i=1,\ldots,n$ and $j=0,1,\ldots n-1$. Hence we  have 
\[
P(x,\partial+\alpha x+\beta)\sim \partial^{n},
\]
i.e.,
\[
P(x,\partial)\sim (\partial-\alpha x-\beta)^{n}.
\]
\end{proof}
\begin{cor}\label{reducible scheme}
Let us take $P(x,\partial)\in W[x,\xi]$ as in Definition \ref{local datum} with the table of local datum,
\begin{center}
\begin{tabular}{|c|c||c|c|}
\hline
\multicolumn{2}{|c||}{}&\multicolumn{2}{|c|}{$\alpha$}\\
\hline
$c_{1}$&$[\mu^{1};m^{1}]$&$\beta$&$[(\nu);(n)]$\\
$c_{2}$&$[\mu^{2};m^{2}]$&&\\
$\vdots$&$\vdots$&&\\
 $c_{p}$&$[\mu^{p};m^{p}]$&&\\
\hline
\end{tabular}.
\end{center}
And we assume that $n=\sum_{i=1}^{r}\sum_{j=1}^{s_{i}}m^{i}_{j}$.
Then we have
\[
E(\frac{\alpha}{2}x^{2};\beta,\nu)P\sim (\partial-\alpha x-\beta)^{n}
\]
\end{cor}
\begin{proof}
By Theorem \ref{addition and middle convolution}, we can see that
 $E(\frac{\alpha}{2}x^{2};\beta,\nu)P$ has the table of local datum
\begin{center}
\begin{tabular}{|c|c||c|c|}
\hline
\multicolumn{2}{|c||}{}&\multicolumn{2}{|c|}{$\alpha$}\\
\hline
$c_{1}$&$[\mu^{1};m^{1}]$&$\beta$&$[(0);(n)]$\\
$c_{2}$&$[\mu^{2};m^{2}]$&&\\
$\vdots$&$\vdots$&&\\
 $c_{p}$&$[\mu^{p};m^{p}]$&&\\
\hline
\end{tabular}.
\end{center}
Therefore by Lemma \ref{reducibility}, we have
\[
 E(\frac{\alpha}{2}x^{2};\beta,\nu)P\sim (\partial-\alpha x-\beta)^{n}.
\]

\end{proof}
\begin{prop}\label{inversion of Euler transform}
Let us take $P(x,\partial)\in W[x,\xi]$ as in Definition \ref{local
 datum} and assume that $P(x,\partial)$ has the nontrivial table of local datum
\begin{center}
\begin{tabular}{|c|c||c|c|c|c|c|c|c|}
\hline
\multicolumn{2}{|c||}{}&\multicolumn{2}{|c|}{$\alpha_{1}$}&\multicolumn{2}{c|}{$\alpha_{2}$}&$\cdots$&\multicolumn{2}{c|}{$\alpha_{r}$}\\
\hline
$c_{1}$&$[\mu^{1};m^{1}]$&$\beta^{1}_{1}$&$[\nu^{11};n^{11}]$&$\beta^{2}_{1}$&$[\nu^{21};n^{21}]$&$\cdots$&$\beta^{r}_{1}$&$[\nu^{r1};n^{r1}]$\\

$c_{2}$&$[\mu^{2};m^{2}]$&$\beta^{1}_{2}$&$[\nu^{12};n^{12}]$&$\beta^{2}_{2}$&$[\nu^{22};n^{22}]$&$\cdots$&$\beta^{r}_{2}$&$[\nu^{r2};n^{r2}]$\\
$\vdots$&$\vdots$&$\vdots$&$\vdots$&$\vdots$&$\vdots$&&$\vdots$&$\vdots$\\
 $c_{p}$&$[\mu^{p};m^{p}]$&$\beta^{1}_{l_{1}}$&$[\nu^{1l_{1}};n^{1l_{1}}]$&$\beta^{2}_{l_{2}}$&$[\nu^{2l_{2}};n^{2l_{2}}]$&&$\beta^{r}_{l_{r}}$&$[\nu^{rl_{r}};n^{rl_{r}}]$\\
\hline
\end{tabular}.
\end{center}Then we have the followings.
\begin{enumerate}
\item We have 
\[
E(\frac{\alpha_{i}}{2}x^{2};*,0)P\sim P.
\]
for $i=1,\ldots,r$ and for any complex number $*$.
\item Let us take $i\in\{1,\ldots,r\}$ and $j\in\{1,\ldots,l_{i}\}$ and
      fix them. If $\nu^{ij}_{k}-f(\xi)\notin\Z\backslash\{1\}$ for all $k=1,\ldots,t_{ij}$,
      then we have
\[
E(\frac{\alpha_{i}}{2}x^{2};\beta^{i}_{j};-f(\xi))E(\frac{\alpha_{i}}{2}x^{2};\-\beta^{i}_{j};f(\xi))P\sim P.
\]
\end{enumerate}
\end{prop}
\begin{proof}
By the assumption, we can write
\[
\mathrm{R}P(x,\partial+\alpha_{i}x)=x^{N}\prod_{j=1}^{l_{i}}(\partial-\beta^{i}_{j})^{n^{i}_{j}}+\sum_{k=1}^{N}x^{N-k}\prod_{j=1}^{l_{i}}(\partial-\beta^{i}_{j})^{\max\{n^{i}_{j}-k\}}P_{k}(\partial)
\]
for $P_{k}(x)\in \C[x]$ and $N=n-n_{i}+\sum_{k=1}^{p}\sum_{l=1}^{s_{k}}m^{k}_{l}$. Hence we can apply Proposition \ref{inverse of euler transform}.
\end{proof}
Although the twisted Euler transform in Corollary \ref{reducible scheme} does not
satisfy the assumption of 2 in Proposition \ref{inversion of Euler
transform}, we can show the following.
\begin{prop}
Let us take $P(x,\partial)$ as in Definition \ref{local datum} with the
 table of local datum,
\begin{center}
\begin{tabular}{|c|c||c|c|}
\hline
\multicolumn{2}{|c||}{}&\multicolumn{2}{|c|}{$\alpha$}\\
\hline
$c_{1}$&$[\mu^{1};m^{1}]$&$\beta$&$[(\nu);(n)]$\\
$c_{2}$&$[\mu^{2};m^{2}]$&&\\
$\vdots$&$\vdots$&&\\
 $c_{p}$&$[\mu^{p};m^{p}]$&&\\
\hline
\end{tabular}.
\end{center}
Then we have
\[
 E(\frac{\alpha}{2}x^{2};\beta,-\nu)E(\frac{\alpha}{2}x^{2};\beta,\nu)P\sim P.
\]
\end{prop}
\begin{proof}
Without loss of the generality, we can assume $\alpha=0$. The table of local datum of
 $\mathrm{RAd}((x+\beta)^{-\nu})\mathcal{L}\mathrm{R}P$ is 
\begin{center}
\begin{tabular}{|c|c||c|c|}
\hline
\multicolumn{2}{|c||}{}&\multicolumn{2}{|c|}{0}\\
\hline
$-\beta$&$[(0,-1);(N,n)]$&$c_{1}$&$[\mu'^{1};m^{1}]$\\
&&$c_{2}$&$[\mu'^{2};m^{1}]$\\
&&$\vdots$&$\vdots$\\
&& $c_{p}$&$[\mu'^{p};m^{p}]$\\
\hline
\end{tabular}.
\end{center}
Here $N=\sum_{i=1}^{r}\sum_{j=1}^{s_{i}}m^{i}_{j}-n$ and 
\[
 \mu'^{i}=(\mu^{i}_{1}+\nu+1,\ldots,\mu^{i}_{s_{i}}+\nu+1)
\]
for $i=1,\ldots,r$. Hence by the same
 argument as in Proposition \ref{inverse of euler transform}, we can
 show that
\[
 (\mathcal{L}^{-1}\mathrm{R}\mathcal{L})\mathrm{RAd}((x+\beta)^{-\nu})\mathcal{L}^{-1}P\sim \mathrm{RAd}((x+\beta)^{-\nu})\mathcal{L}^{-1}P.
\] 
Hence we have
\begin{align*}
 E(\beta,-\nu)E(\beta,\nu)P&=\mathcal{L}\mathrm{RAd}((x+\beta)^{\nu})\mathcal{L}^{-1}\mathrm{R}\mathcal{L}\mathrm{RAd}((x+\beta)^{-\nu})\mathcal{L}^{-1}\mathrm{R}P\\
&=\mathcal{L}\mathrm{RAd}((x+\beta)^{\nu})\mathrm{RAd}((x+\beta)^{-\nu})\mathcal{L}^{-1}P\\
&=\mathcal{L}\mathrm{R}\mathcal{L}^{-1}P=E(*,0)P\\
&\sim P.
\end{align*}
\end{proof}
\section{Kac-Moody root system}
P. Boalch found a correspondence between quiver varieties and moduli
spaces of  meromorphic
connections on  vector bundles over the  Riemann sphere of the forms
\[
 (\frac{A}{z^{3}}+\frac{B}{z^{2}}+\frac{C}{z})dz.
\]
He studied the existence of these meromorphic connections through the
theory of representations of quiver varieties which is first studied by W.
Crawley-Boevey in \cite{C-B1}.

In this section, as an analogue of this result of Boalch, we attach a
differential operator considered in Definition \ref{local datum} to  a
Kac-Moody Lie algebra and an element of the root lattice  of this algebra. And we show
the equivalence between twisted Euler transforms and additions on
differential equations and the action of Weyl group on the corresponding
element of the root lattice.

Let us take $P(x,\partial)\in W[x,\xi]$ with the table of local datum
\begin{center}
\begin{tabular}{|c|c||c|c|c|c|c|c|c|}
\hline
\multicolumn{2}{|c||}{}&\multicolumn{2}{|c|}{$\alpha_{1}$}&\multicolumn{2}{c|}{$\alpha_{2}$}&$\cdots$&\multicolumn{2}{c|}{$\alpha_{r}$}\\
\hline
$c_{1}$&$[\mu^{1};m^{1}]$&$\beta^{1}_{1}$&$[\nu^{11};n^{11}]$&$\beta^{2}_{1}$&$[\nu^{21};n^{21}]$&$\cdots$&$\beta^{r}_{1}$&$[\nu^{r1};n^{r1}]$\\

$c_{2}$&$[\mu^{2};m^{2}]$&$\beta^{1}_{2}$&$[\nu^{12};n^{12}]$&$\beta^{2}_{2}$&$[\nu^{22};n^{22}]$&$\cdots$&$\beta^{r}_{2}$&$[\nu^{r2};n^{r2}]$\\
$\vdots$&$\vdots$&$\vdots$&$\vdots$&$\vdots$&$\vdots$&&$\vdots$&$\vdots$\\
 $c_{p}$&$[\mu^{p};m^{p}]$&$\beta^{1}_{l_{1}}$&$[\nu^{1l_{1}};n^{1l_{1}}]$&$\beta^{2}_{l_{2}}$&$[\nu^{2l_{2}};n^{2l_{2}}]$&&$\beta^{r}_{l_{r}}$&$[\nu^{rl_{r}};n^{rl_{r}}]$\\
\hline
\end{tabular}.
\end{center}
We fix this $P(x,\partial)$ through this section.

Let $\mathfrak{h}$ be the complex vector space with the basis 
\[
\Pi=\{v^{ij}_{k}\mid i=0,\ldots,r, j=1,\ldots, l_{i}, k=1,\ldots,t^{ij}\}.
\]
Here we put $l_{0}=p$ and $t^{0j}=s_{j}$. We define the non-degenerate symmetric bilinear form on $\mathfrak{h}$ as follows,
\[
(v^{ij}_{k},v^{lm}_{n})=
\begin{cases}
2&\text{if }(i,j,k)=(l,m,n)\\
-1&\text{if }(i,j)=(l,m)\text{ and } |k-n| =1\\
-1&\text{if }(k,n)=(1,1)\text{ and } i\neq j\\
0&\text{otherwise}
\end{cases}.
\]
 If for $s=(i_{s},j_{s},k_{s})\in I=\{(i,j,k) \mid i=0,\ldots,r, j=1,\ldots, l_{i}, k=1,\ldots,t^{ij}\}$, we write $v_{s}=v^{i_{s}j_{s}}_{k_{s}}$. We can define the generalized Cartan matrix,
 \[
 A=\left(\frac{2(v_{s},v_{t})}{(v_{s},v_{s})}\right)_{s\in I,t\in I}.
 \]
Let $\mathfrak{g}(A)$ be the Kac-Moody Lie algebra with the above generalized Cartan matrix $A$. According to the usual terminology, we call $\Pi$ the root basis, elements from $\Pi$ are called simple roots and $\Z$-lattice generated by $\Pi$, i.e.
\[
Q=\sum_{(i,j,k)\in I}\Z v^{ij}_{k}
\]
is called the root lattice. Also we define the positive root lattice
$$Q^{+}=\sum_{(i,j,k)\in I}\Z_{\ge 0} v^{ij}_{k}.$$ The height of an
element of the root lattice $\alpha=\sum_{i,j,k\in
I}x^{ij}_{k}v^{ij}_{k}$ is defined by
\[
 \mathrm{ht}(\alpha)=\sum_{(i,j,k)\in I}x^{ij}_{k}.
\]
Also the support of $\alpha$ is defined by
\[
 \supp\alpha=\{v^{ij}_{k}\in\Pi\mid x^{ij}_{k}\neq 0\}.
\]
We say that the subset $L\subset\Pi$ is connected if the decomposition
$L_{1}\cup L_{2}=L$ with $L_{1}\neq \emptyset$ and $L_{2}\neq \emptyset$
always implies the existence of $v_{i}\in L_{i}$ satisfying
$(v_{1},v_{2})\neq 0$.

We have the following root space decomposition of $\g(A)$ with respect to $\mathfrak{h}$,
\[
\g(A)=\bigoplus_{\alpha\in Q}\g_{\alpha}
\]
where $\g_{\alpha}=\{X\in\g(A)\mid [H,X]=(\alpha,H)X\text{ for all }H\in\mathfrak{h}\}$ is the root space attached to $\alpha$. The root space is $\Delta=\{\alpha\in Q\mid \g_{\alpha}\neq\{0\}\}$ and we call elements of $\Delta$ roots. 

The reflections on $\mathfrak{h}$ with respect to simple roots $v^{ij}_{k}$, so-called simple reflections, are defined by
\[
r^{ij}_{k}\colon \mathfrak{h}\ni H\longmapsto r^{ij}_{k}(H)=H-\frac{2(H,v^{ij}_{k})}{(v^{ij}_{k},v^{ij}_{k})}v^{ij}_{k}=H-(H,v^{ij}_{k})v^{ij}_{k}.
\]
The Weyl group $W$ is the group generated by all simple reflections. 

A root $\alpha\in \Delta$ is called real root if there exists $w\in W$
such that $w(\alpha)\in \Pi$. We denote the set of real roots by
$\Delta^{re}$. A root which is not real root is called imaginary
root. We denote the set of all imaginary roots by $\Delta^{im}$. Hence
there is a decomposition of the set of roots, $\Delta =\Delta^{re}\cup \Delta^{im}.$ If the Cartan matrix is symmetrizable, we can see that 
\begin{align*}
&\Delta^{re}=\{\alpha\in\Delta\mid (\alpha,\alpha)>0\},&\Delta^{im}=\{\alpha\in\Delta\mid (\alpha,\alpha)\le 0\}.
\end{align*}
In our case the Cartan matrix $A$ is symmetric. Hence moreover we have 
\[
\Delta^{re}=\{\alpha\in\Delta\mid (\alpha,\alpha)=2\}.
\]

For fundamental things about Kac-Moody Lie algebra, we refer the
standard text book \cite{Kac}.
 
For the above $P\in W[x,\xi]$, we can define the element $\alpha_{P}\in
Q^{+}$ associated with $P(x,\partial)$ as follows. Let us put
\[
\tilde{n}^{ij}_{k}=
\begin{cases}
\sum_{l=k}^{s_{j}}m^{j}_{l}&\text{if } i=0\\
\sum_{l=k}^{t_{ij}}n^{ij}_{l}&\text{if } i=1,\ldots,r 
\end{cases}.
\]
Then $\alpha_{P}\in Q^{+}$ is defined by 
\[
\alpha_{P}=\sum_{(i,j,k)\in I}\tilde{n}^{ij}_{k}v^{ij}_{k}.
\]
\begin{rem}This correspondence between $P\in W[x,\xi]$ and
 $\alpha_{P}\in Q^{+}$ is not unique. Indeed,  for $\nu^{ij}_{k}$
and $n^{ij}_{k}$, permutations with respect to the index $k=1,\ldots,t^{ij}$ do not change local solutions of $P$. For $\mu^{i}_{j}$ and $m^{i}_{j}$, permutations with respect to the index $j=1,\ldots,s_{i}$ do not change local solutions of $P$ as well.  

\end{rem}
\begin{exa}

If $P(x,\partial)$ has the table of local datum,
\begin{center}
\begin{tabular}{|c|c||c|c|c|c|}
\hline
\multicolumn{2}{|c||}{}&\multicolumn{2}{|c|}{$\alpha_{1}$}&\multicolumn{2}{|c|}{$\alpha_{2}$}\\
\hline
$c_{1}$&$[(\mu^{1}_{1});(1)]$&$\beta_{1}$&$[(\nu^{1}_{1});(1)]$&$\beta_{2}$&$[(\nu^{2}_{1});(1)]$\\
\hline
\end{tabular},
\end{center}
the corresponding Kac-Moody Lie algebra has the following dynkin
 diagram,
\begin{center}
\includegraphics{test.1}.
\end{center}
If $P(x,\partial)$ has the table of local datum,
\begin{center}
\begin{tabular}{|c|c||c|c|}
\hline
\multicolumn{2}{|c||}{}&\multicolumn{2}{|c|}{$\alpha$}\\
\hline
$c_{1}$&$[(\mu^{1}_{1},\mu^{1}_{2});(1,1)]$&$\beta_{1}$&$[(\nu^{1}_{1},\nu^{2}_{1});(1,1)]$\\
$c_{2}$&$[(\mu^{2}_{1},\mu^{2}_{1});(1,1)]$&$\beta_{2}$&$[(\nu^{2}_{1},\nu^{2}_{2});(1,1)]$\\
&&$\beta_{3}$&$[(\nu^{3}_{1},\nu^{3}_{1});(1,1)]$\\
\hline
\end{tabular},
\end{center}
the corresponding Kac-Moody Lie algebra has the following dynkin
 diagram,
\begin{center}
\includegraphics{test2.1}.
\end{center}
\end{exa}
\begin{thm}\label{reflection middle convolution}
We retain the above notations. 
\begin{enumerate}
\item For $P\in W[x,\xi]$ and $\alpha_{P}\in Q^{+}$ defined as above, we have
\[
\mathrm{idx}\,P=(\alpha_{P},\alpha_{P}).
\]
\item Let us assume $n-n^{ij}_{k}>0$ or $N_{i}-n^{i}_{j}>0$ where
      $N_{i}=\sum_{k=1}^{l_{0}}\sum_{l=1}^{t_{0k}}m^{k}_{l}+(n-n^{i})$. If we put $Q=E(\frac{\alpha_{i}}{2}x^{2}+\beta^{i}_{j}x;\beta^{i}_{j},\nu^{ij}_{1}-1)P(x,\partial)$, then we have
\[
\alpha_{Q}=r^{ij}_{1}(\alpha_{P}).
\]
\item If we put $Q=\mathrm{RAd}((x-c_{i})^{-\mu^{i}_{1}})P$, then we have
\[
\alpha_{Q}=r^{0i}_{1}(\alpha_{P}).
\]
\item For $k\ge2,$  reflections $r^{ij}_{k}(\alpha_{P})$ correspond
      following permutations of the table of local datum  of $P$,
\[
(\mu^{y}_{z},m^{y}_{z})\mapsto
\begin{cases}
(\mu^{y}_{z+1},m^{y}_{z+1})&\text{if } i=0, j=y \text{ and }k=z\\
(\mu^{y}_{z+1},m^{y}_{z-1})&\text{if } i=0, j=y \text{ and }k=z-1\\
(\mu^{y}_{z},m^{y}_{z})&\text{otherwise}
\end{cases} ,
\]
\[
(\nu^{xy}_{z},n^{xy}_{z})\mapsto
\begin{cases}
(\nu^{xy}_{z+1},n^{xy}_{z+1})& \text{if } (x,y,z)=(i,j,k)\\
(\nu^{xy}_{z-1},n^{xy}_{z-1})& \text{if } (x,y,z-1)=(i,j,k)\\
(\nu^{xy}_{z},n^{xy}_{z})& \text{otherwise } 
\end{cases}.
\]
Here $\mu^{y}_{z}$ and $\nu^{xy}_{z}$ are exponents of local solutions of $P$ and $m^{y}_{z}$ and $n^{xy}_{z}$ are multiplicities of them respectively.
\end{enumerate}
\end{thm}
This theorem immediately follows from the following lemma.
\begin{lem}
Let us take $\alpha\in \mathfrak{h}$ such that 
\[
\alpha=\sum_{(i,j,k)\in I}c^{ij}_{k}v^{ij}_{k}
\]
for $c^{ij}_{k}\in \C$. The we have the following equations,
\begin{align*}
(\alpha,\alpha)=&-(l_{0}+1)(\sum_{i=1}^{r}\sum_{j=1}^{l_{i}}c^{ij}_{1})^{2}+\sum_{i=1}^{r}(\sum_{j=1}^{l_{i}}c^{ij}_{1})^{2}+\sum_{i=1}^{r}\sum_{j=1}^{l_{i}}(c^{ij}_{1})^{2}\\
&+\sum_{i=1}^{l_{0}}(\sum_{j=1}^{r}\sum_{k=1}^{l_{j}}c^{jk}_{1}-c^{0i}_{1})+\sum_{i=0}^{r}\sum_{j=1}^{l_{i}}\sum_{k=1}^{t_{ij}}(c_{k}^{ij}-c_{k+1}^{ij})^{2},
\end{align*}
\begin{multline*}
(\alpha,v^{i_{0}j_{0}}_{k_{0}})=\\
\begin{cases}
-(\displaystyle\sum_{i=1}^{r}\sum_{j=1}^{l_{i}}\sum_{k=1}^{t_{ij}}c^{ij}_{k}-\sum_{i=1}^{l_{i_{0}}}\sum_{j=1}^{t_{i_{0}i}}c^{i_{0}j}_{k})+c^{i_{0}j_{0}}_{1}+(c^{i_{0}j_{0}}_{1}-c^{i_{0}j_{0}}_{2})&\text{if } k_{0}=1\\
(c^{i_{0}j_{0}}_{k_{0}}-c^{i_{0}j_{0}}_{k_{0}-1})+(c^{i_{0}j_{0}}_{k_{0}}-c^{i_{0}j_{0}}_{k_{0}+1})&\text{if }k_{0}\ge 2
\end{cases}.
\end{multline*}
\end{lem}
\begin{proof}
Direct computation.
\end{proof}

We consider a subset of $Q^{+}$,
\[
V=\{nv^{i'j'}_{1}+\sum_{j=1}^{l_{0}}\sum_{k=1}^{n}m^{j}_{k}v^{0j}_{k}\mid \begin{subarray}{c}1\le i'\le r,1\le j'\le l_{i'}\\
 n,m^{j}_{k}\in\Z_{>0}\text { such that }n\ge m^{j}_{1}>m^{j}_{2}>\cdots\end{subarray}\}
\]
\begin{prop}
If $\alpha_{P}\notin W(V)$, the set of all Weyl group orbits of elements
 of $V$, then the Weyl group orbit of $\alpha_{P}$, $W(\alpha_{P})$, is
 contained in $Q^{+}$.
\end{prop}
\begin{proof}
If $\alpha_{P}\notin V$, a twisted Euler transform of $P(x,\partial)$
 corresponds to a simple reflection of $\alpha_{P}$ by Theorem
 \ref{reflection middle convolution}. Hence  we can find
 $Q(x,\partial)\in W[x,\xi]$ which corresponds to
 $\alpha_{Q}=r^{ij}_{1}(\alpha_{P})$ by taking the twisted Euler
 transform of $P(x,\partial)$.  Moreover if $\alpha_{Q}\notin V$, we can
 find $Q'(x,\partial)\in W[x,\xi]$ which corresponds to $\alpha_{Q'}=r^{i'j'}_{1}(\alpha_{Q})=r^{i'j'}_{1}r^{ij}_{1}(\alpha_{P})$ by the twisted Euler transform. Hence if $\alpha_{P}\notin W(V)$, we can iterate these. Also we can use same argument for other simple reflections. Then we have the proposition.
\end{proof}
\begin{cor}\label{imaginary}
If $\alpha_{P}\notin W(V)$, then $\alpha_{P}\in \Delta^{im}$.
\end{cor}
\begin{proof}
First we assume that $(\alpha_{P},\alpha_{P})>0$. Let $\beta$ be an
 element of minimal height among $W(\alpha_{P})\cap Q^{+}$. Since the Weyl group action does not change inner product, we have $(\beta,\beta)>0$. Hence we have $(\beta,v^{ij}_{k})>0$ for some $(i,j,j)\in I$. If $\beta\neq v^{ij}_{k}$, then $r^{ij}_{k}(\beta)\in Q^{+}$ and $\mathrm{ht}(r^{ij}_{k}(\beta))<\mathrm{ht}(\beta)$, a contradiction with the choice of $\beta$. Hence $\beta=v^{ij}_{k}$. Since $\alpha_{P}\notin W(V)$, we can find $Q_{\beta}\in W[x,\xi]$ such that $\alpha_{Q_{\beta}}=\beta$. This implies that if $\beta$ is the simple root, then $\beta=v^{ij}_{1}$ for some $i$ and $j$.  However $v^{ij}_{1}\in V$. This contradicts our assumption. Hence if $\alpha_{P}\notin W(V)$, we have $(\alpha_{P},\alpha_{P})\le 0$.

Now we assume that $(\alpha_{P},\alpha_{P})\le 0$. As above, we choose
 an element $\beta\in W(\alpha_{P})$ of minimal height. Then
 $(\beta,v^{ij}_{k})\le 0$ for all $(i,j,k)\in I$. Since $\beta$
 corresponds to some $Q_{\beta}\in W[x,\xi]$, this implies that $\supp\beta$ is connected. Hence \[
\beta\in K=\{\alpha\in Q^{+}\backslash \{0\}\mid \substack{(\alpha,v^{ij}_{k})\le 0\text{ for all $(i,j,k)\in I$}\\\text{ and $\mathrm{supp}\alpha$ is connected}}\}.
\]
This implies that $\alpha_{P}\in W(K)=\Delta^{im}\cap Q^{+}$.
\end{proof}
\begin{thm}\label{rigid reduction}
Let  us take $P(x,\partial)\in W[x,\xi]$ as in Definition \ref{local
 datum}. If $\mathrm{idx}\,P>0$, then $P(x,\partial)$ can be reduced to 
\[
(\partial-\alpha x-\beta)^{n}
\]
for some $\alpha,\beta\in \C$ and $n\in \Z_{>0}$ by finite iterations of
 twisted Euler transforms and additions at regular singular points. 
\end{thm}
\begin{proof}
By Corollary \ref{imaginary}, if
 $\mathrm{idx}\,P=(\alpha_{P},\alpha_{P})>0$, then $\alpha_{P}\in
 W(V)$. Hence finite iterations of simple reflections $\alpha_{P}$
 reduces to an element of $V$. This implies $P$ reduces to a $Q\in
 W[x,\xi]$ with a table of local datum
\begin{center}
\begin{tabular}{|c|c||c|c|}
\hline
\multicolumn{2}{|c||}{}&\multicolumn{2}{|c|}{$\alpha$}\\
\hline
$c_{1}$&$[\mu^{1};m^{1}]$&$\beta$&$[(\nu);(n)]$\\
$c_{2}$&$[\mu^{2};m^{2}]$&&\\
$\vdots$&$\vdots$&&\\
 $c_{p}$&$[\mu^{p};m^{p}]$&&\\
\hline
\end{tabular},
\end{center}
by finite iterations of twisted Euler transforms  and additions at regular
 singular points. Proposition \ref{reducible scheme} says that 
\[
 E(\frac{\alpha}{2}x^{2};\beta,\nu)Q\sim (\partial-\alpha x-\beta)^{n}.
\]
Hence we have the theorem.
\end{proof}
\section{Confluence}
In this section, we show that differential operator $P(x,\partial)$ of
$\mathrm{idx}P>0$ can be obtained by the limit transition from a
Fuchsian differential operator of the same rigidity index.
\subsection{Fuchsian differential equations}
\begin{df}
Let us take  $P(x,\xi)\in W[x,\xi]$ as in Definition \ref{local datum}. If $P(x,\partial)$ has the following table of local datum,

\begin{center}
\begin{tabular}{|c|c||c|c|}
\hline
\multicolumn{2}{|c||}{}&\multicolumn{2}{|c|}{$0$}\\
\hline
$c_{1}$&$[\mu^{1};m^{1}]$&$0$&$[(\nu_{1},\ldots,\nu_{t});(n_{1},\ldots,n_{t})]$\\
$c_{2}$&$[\mu^{2};m^{2}]$&&\\
$\vdots$&$\vdots$&&\\
 $c_{p}$&$[\mu^{p};m^{p}]$&&\\
\hline
\end{tabular},
\end{center}
then we say that $P(x,\partial)$ is Fuchsian.
\end{df}

\begin{prop}
If $P(x,\partial)$ is Fuchsian with nontrivial table of local datum, then $E(0,f(\xi))P$ and $\mathrm{RAd}((x-c)^{g(\xi)}) P$ are Fuchsian for any $c\in\C$ and $f(\xi),g(\xi)\in \C(\xi)$.

\end{prop}
\begin{proof}
This is a collorary of Theorem \ref{addition and middle convolution}.
\end{proof}
\subsection{Versal additions}
We define the operator called versal additions. These operators are
introduced by Oshima in \cite{O}.

For $a_{1},\ldots,a_{n}\in \C$, we define a function
\[
h_{n}(a_{1},\ldots,c_{n};x)=-\int_{0}^{x}\frac{t^{n-1}\,dt}{\prod_{1\le i\le n}(1-a_{i}t)}.
\]
Then it is not hard to see that 
\[
e^{\lambda_{n}h_{n}(a_{1},\ldots,a_{n};x)}=\prod_{k=1}^{n}(1-a_{k}x)^{\frac{\lambda_{n}}{a_{k}\prod_{\substack{1\le i\le n\\i\neq k}}(a_{k}-a_{i})}}.
\]
\begin{df}[Versal addition]
We put
\begin{align*} 
\mathrm{AdV}(a_{1},\ldots,a_{n};\lambda_{1},\ldots,\lambda_{n})&=\prod_{k=1}^{n}\mathrm{Ad}\left((x-\frac{1}{a_{k}})^{\sum_{l=k}^{n}\frac{\lambda_{l}}{a_{k}\prod_{\substack{1\le i\le l\\i\neq k}}(a_{k}-a_{i})}}\right).\\
\end{align*}
\end{df}
\begin{prop}\label{limit transition}
For $P(x,\partial)\in W[x]$, we have
\begin{align*}
&\lim_{a_{1}\rightarrow 0}\mathrm{AdV}(a_{1};\lambda_{1})P(x,\partial)=\mathrm{Ade}(-\lambda_{1}x)P(x,\partial),\\
&\lim_{\substack{a_{1}\rightarrow 0\\a_{2}\rightarrow 0}}\mathrm{AdV}(a_{1},a_{2};\lambda_{1},\lambda_{2})P(x,\partial)=\mathrm{Ade}(\lambda_{1}x^{2}+\lambda_{2}x)P(x,\partial).
\end{align*}
\end{prop}
\begin{proof}
If we recall that
\[
 \mathrm{AdV}(a_{1};\lambda_{1})\colon \partial=
 \partial-\frac{\frac{\lambda_{1}}{a_{1}}}{(x-\frac{1}{a_{1}})}=\partial+\frac{\lambda_{1}}{1-a_{1}x},
\]
then we can see that
\[
 \lim_{a_{1}\rightarrow
 0}{AdV}(a_{1};\lambda_{1})\partial=\partial+\lambda_{1}=\mathrm{Ade}(-\lambda_{1}x)\partial.
\]
Also the equation
\begin{align*}
 \mathrm{AdV}(a_{1},a_{2};\lambda_{1},\lambda_{2})\partial&=\partial-\frac{\frac{\lambda_{1}}{a_{1}}+\frac{\lambda_{2}}{a_{1}(a_{1}-a_{2})}}{x-\frac{1}{a_{1}}}-\frac{\frac{\lambda_{2}}{a_{2}(a_{2}-a_{1})}}{x-\frac{1}{a_{2}}}\\&
=\partial+\frac{\lambda_{1}+\frac{\lambda_{2}}{(a_{1}-a_{2})}}{1-a_{1}x}+\frac{\frac{\lambda_{2}}{(a_{2}-a_{1})}}{1-a_{2}x}\\
&=\partial+\frac{\lambda_{1}}{1-a_{1}x}+\frac{\lambda_{2}x}{(1-a_{1}x)(1-a_{2}x)},
\end{align*}
 implies that 
\[
 \lim_{\substack{a_{1}\rightarrow 0\\a_{2}\rightarrow
 0}}\mathrm{AdV}(a_{1},a_{2};\lambda_{1},\lambda_{2})\partial=\partial+\lambda_{1}+\lambda_{2}x=\mathrm{Ade}(-\frac{\lambda_{2}}{2}x^{2}-\lambda_{1}x)\partial.
\]
\end{proof}
\begin{thm}
Take a $P(x,\partial)\in W[x,\xi]$ as in Definition \ref{local datum}.  If $\mathrm{idx}\,P>0$, then $P(x,\partial)$ can be obtained by  the limit transition of a Fuchsian $Q(x,\partial)\in W[x,\xi]$ of $\mathrm{idx}\,Q=\mathrm{idx}\,P.$
\end{thm}
\begin{proof}
By Theorem \ref{rigid reduction}, $P(x,\partial)$ is obtained by finite iterations of twisted Euler transforms and additions  from $(\partial-\alpha x-\beta)^{n}$. Here $\partial-\alpha x-\beta=\lim_{\substack{a_{1}\rightarrow 0\\a_{2}\rightarrow 0}}\mathrm{AdV}(a_{1},a_{2};-2\alpha,-\beta)\partial.$ As we see in Remark \ref{Euler and twisted Euler}, twisted Euler transforms are compositions of Euler transforms, $\mathrm{Ade}(\alpha x)$ and $\mathrm{Ade}(\beta x^{2}+\gamma x)$ for some $\alpha,\beta,\gamma\in\C$. Hence twisted Euler transforms can be obtained by the limit transitions of compositions of additions and Euler transforms by Proposition \ref{limit transition}. 

Therefore $P(x,\partial)$ can be seen as the limit of a Fuchsian
 $Q(x,\xi)$ which is obtained by Euler transforms and additions from
 $\mathrm{AdV}(a_{1},a_{2};-2\alpha,-\beta)\partial.$ 

Finally we notice
 that twisted Euler transforms do not change rigidity indices because
 the action of Weyl group does not change the inner product.
\end{proof}
\appendix
\section*{Appendix}
\section{Differential equations with regular singularity at $x=\infty$ and Euler transform}
We consider differential equations with regular singular point at $x=\infty$ and arbitrary singularities at any other points in $\C$. And then we give a necessary and sufficient condition to reduce the rank of differential equation by Euler transform.
\begin{thm}
Let us take $P(x,\partial)\in W[x]$ which has regular singular point at $x=\infty$ and semi-simple exponents
\[
\{[\mu_{1}]_{n_{1}},\ldots,[\mu_{l}]_{n_{l}}\},
\] 
where $\sum_{i=1}^{l}n_{i}=n=\mathrm{ord}\,P$, $\mu_{i}\notin\Z$ and $\mu_{i}-\mu_{j}\notin\Z$ if $i\neq j$. Then we have
\[
\mathrm{ord}\,E(0,\mu_{i}-1)P(x,\partial)<\mathrm{ord}\,P
\]
if and only if 
\[
\deg P-\mathrm{ord}\, P<n_{i}.
\]
\begin{proof}
Since $x=\infty$ is the regular singular point of $P(x,\partial)\in W[x]$, we can write
\[
\mathrm{R}P(x,\partial)=\sum_{i=0}^{N}x^{N-i}\partial^{\max\{n-i,0\}}P_{i}(\partial)
\]
for $P_{i}(x)\in\C[x]$ of $\deg P_{i}\le n$ for $i=0,\ldots, N$ and $P_{0}(x)=1$. Here $N=\deg P$ and $n=\mathrm{ord}\, P$. Hence we have
\[
\mathcal{L}^{-1}\mathrm{R}P=\sum_{i=0}^{N}\partial^{N-i}(-x)^{\max\{n-i\}}P_{i}(-x)
\]
and this has regular singular point at $x=0$ and no other singular points in $\C$. Also this has  semi-simple exponents,
\[
\{[0]_{N-n},[\mu_{1}-1]_{n_{1}},\ldots,[\mu_{l}-1]_{n_{l}}\}
\]
at $x=0$ by Proposition {laplace inverse transform of regular point}. And then we can see that $\mathrm{RAd}(x^{-\mu_{i}+1})\mathcal{L}^{-1}\mathrm{R}P$ has semi-simple exponents,
\[
\{[-\mu_{i}+1]_{N-n},[\mu_{1}-\mu_{i}]_{n_{1}},\ldots,[\mu_{i-1}-\mu_{i}]_{n_{i-1}},[0]_{n_{i}}[\mu_{i+1}-\mu_{i}]_{n_{i}},\ldots\}
\]
by Proposition \ref{addition and exponents}. Hence by Proposition \ref{reduce degree}, we have
\[
\deg \mathrm{RAd}(x^{-\mu_{i}+1})\mathcal{L}^{-1}\mathrm{R}P=n-n_{i}+(N-n).
\]
 This means that 
\begin{align*}
\mathrm{ord}\, \mathcal{L} \mathrm{RAd}(x^{-\mu_{i}+1})\mathcal{L}^{-1}\mathrm{R}P&=E(0,\mu_{i}-1)P\\
&=n-n_{i}+(N-n).
\end{align*}
Hence we have the theorem.

\end{proof}
\end{thm}

\end{document}